\documentclass[12pt]{amsart}
\usepackage{import}
\usepackage{comment}
\usepackage{mathrsfs}
\usepackage[margin=1in]{geometry}
\usepackage{subcaption}
\usepackage{hyperref}
\hypersetup{
    colorlinks=true,
    linkcolor=blue,
    filecolor=blue,      
    urlcolor=blue,
    citecolor=blue,
    pdftitle={Overleaf Example},
    pdfpagemode=FullScreen,
    }

\usepackage{amsmath}
\usepackage{amsfonts}
\usepackage{amssymb}
\usepackage{graphicx}
\usepackage{tikz-cd}
\usepackage{amsthm}

\newtheorem{counter}{counter}[section]
\newtheorem{theorem}[counter]{Theorem}
\newtheorem{corollary}[counter]{Corollary}
\newtheorem{lemma}[counter]{Lemma}
\newtheorem{proposition}[counter]{Proposition}

\theoremstyle{definition}
\newtheorem{definition}[counter]{Definition}
\newtheorem{remark}[counter]{Remark}

\newtheorem{example}[counter]{Example}

\newcommand{\PP}{\mathbb{P}}
\newcommand{\Z}{\mathbb{Z}}
\newcommand{\vv}{\mathbf{v}}
\newcommand{\ww}{\mathbf{w}}
\newcommand{\OO}{\mathcal{O}}

\newcommand{\gr}{\text{gr}}

\newcommand{\Pic}{\text{Pic}}
\newcommand{\ev}{\text{ev}}

\newcommand{\Hom}{\text{Hom}}
\newcommand{\coker}{\text{coker}}
\newcommand{\Ext}{\text{Ext}}
\newcommand{\ext}{\text{ext}}
\newcommand{\Aut}{\text{Aut}}

\newcommand{\ch}{\text{ch}}

\newcommand{\expdim}{\text{expdim}}
\newcommand{\expcodim}{\text{expcodim}}

\newcommand{\Supp}{\text{Supp}}

\newcommand{\Quot}{\text{Quot}}

\title{Higher rank Brill-Noether theory on $\PP^2$}
\author{Ben Gould$^{1, \ast}$, Yeqin Liu$^2$, and Woohyung Lee$^3$}

\begin{document}
\maketitle

\noindent $^{1,2,3}$: Department of Mathematics, Statistics, and Computer Science, University of Illinois at Chicago, Science and Engineering Offices, 851 South Morgan Street, Chicago, IL 60607, USA \\

\noindent $^{\ast}$\textit{Correspondence to be sent to: email: bgould3@uic.edu}

\begin{abstract}
    Let $M_{\PP^2}(\vv)$ be a moduli space of semistable sheaves on $\PP^2$, and let $B^k(\vv) \subseteq M_{\PP^2}(\vv)$ be the \textit{Brill-Noether locus} of sheaves $E$ with $h^0(\PP^2, E) \geq k$. In this paper we develop the foundational properties of Brill-Noether loci on $\PP^2$. Set $r = r(E)$ to be the rank and $c_1, c_2$ the Chern classes. The Brill-Noether loci have natural determinantal scheme structures and expected dimensions $\dim B^k(\vv) = \dim M_{\PP^2}(\vv) - k(k - \chi(E))$. When $c_1 > 0$, we show that the Brill-Noether locus $B^r(\vv)$ is nonempty. When $c_1 = 1$, we show all of the Brill-Noether loci are irreducible and of the expected dimension. We show that when $\mu = c_1/r > 1/2$ is not an integer and $c_2 \gg 0$, the Brill-Noether loci are reducible and describe distinct irreducible components of both expected and unexpected dimension.
\end{abstract}

\tableofcontents

\section{Introduction}
On a polarized variety $(X,H)$ the moduli spaces $M_{X,H}(\vv)$ of $H$-semistable sheaves of numerical type $\vv$ carry \textit{Brill-Noether loci} $B^k(\vv) \subseteq M_{X,H}(\vv)$ whose members $E$ satisfy $h^0(X,E) \geq k$. In this paper we develop the foundational properties of Brill-Noether loci on $X = \PP^2$. Set $r = \ch_0(\vv)$, $\mu = \mu(\vv) = \frac{\ch_1(\mathbf{v})}{\ch_0(\mathbf{v})}$ and $\Delta = \Delta(\vv) = \frac{1}{2}\mu(\mathbf{v})^2 - \frac{\ch_2(\mathbf{v})}{\ch_0(\mathbf{v})}$ to be the rank, slope, and discriminant of $\vv$. The Brill-Noether loci have natural determinantal scheme structures and expected dimensions $\expdim B^k(\vv) = \dim M_{\PP^2}(\vv) - k(k - \chi(E))$. When $\mu > 0$ we show that $B^r(\vv)$ is nonempty. When $\ch_1(\vv) = 1$, we show each Brill-Noether locus $B^k(\vv)$ is irreducible and of the expected dimension as a determinantal variety, $\dim B^k(\vv) = \dim M_{\PP^2}(\vv) - k(k - \chi(\vv))$. We prove that when $r, \ch_1 > 0$ and $\mu > 1/2$ is not an integer with $\Delta \gg 0$, then $B^r(\vv)$ is reducible, and describe distinct irreducible components.

When $X = C$ is a smooth curve and the rank is 1, $M_X(\vv) = \Pic^d(C)$ is the space of line bundles of a given degree $d$, and the Brill-Noether loci $B^k(\vv) = W^k_d(C)$ have been studied for over a century \cite{ACGH}. Much is known about their geometry when $C$ is general: when the expected dimension of $W^k_d(C)$ as a determinantal variety is positive, it is nonempty of that dimension and irreducible. A general $L \in \Pic^d(C)$ has $h^0(L) = \chi(L) = d - g + 1$ when this is non-negative. In higher rank the general bundle still has $h^0(E) = \chi(E)$ when this is non-negative, and the Brill-Noether loci have been studied in detail (for a survey see, e.g., \cite{News} and its bibliography).

On algebraic surfaces, however, much less is known about Brill-Noether theory. The basic theory has begun to be worked out on Hirzebruch surfaces (see \cite{CMR}) and in rank 2 on $\PP^2$ (see \cite{Segre}). On $\PP^2$, the moduli spaces $M_{\PP^2}(\vv)$ of semistable bundles of any rank are irreducible, and a well-known theorem of G\"ottsche-Hirschowitz describes the global sections of a general bundle $E \in M_{\PP^2}(\vv)$ of rank at least two: if the slope $\mu(E)$ is positive, then $h^0(E) = \max\{0, \chi(E)\}$, which is determined by $\vv$. More generally, a general bundle $E$ of any slope has at most one nonzero cohomology group, and by semicontinuity there is an open dense subset of $M_{\PP^2}(\vv)$ of bundles with this cohomology. The Brill-Noether loci $B^k(\vv)$ with $k > \chi(\vv)$ form its complement. These foundational results make the study of the Brill-Noether loci $B^k(\vv)$ on $\PP^2$ approachable.

\subsection{Geometry of Brill-Noether loci}
The moduli spaces $M(\vv) = M_{\PP^2}(\vv)$ of semistable sheaves on $\PP^2$ with Chern character $\vv$ are irreducible projective algebraic varieties. We call a Chern character $\vv \in K(\PP^2)$ stable if it is the Chern character of a stable sheaf, so that $M(\vv) \neq \emptyset$. We are interested in the Brill-Noether loci $B^k(\vv) \subseteq M(\vv)$. As in the case of line bundles on curves, the Brill-Noether loci are constructed as determinantal varieties, so each has an \textit{expected dimension}, which is a lower bound for the dimension of an irreducible component of $B^k(\vv)$; it is
    \[\expdim B^k(\vv) = \dim M_{\PP^2}(\vv) - k(k - \chi(\vv)).\]

Our main results are summarized as follows.

\begin{theorem} \label{main}
Let $\vv \in K(\PP^2)$ be a stable Chern character and $M(\vv)$ the associated moduli space of semistable sheaves. Suppose that $\mu(\vv) > 0$ and set $r = \ch_0(\vv)$. 
\begin{enumerate}
    \item (Emptiness) For any $E \in M(\vv)$, 
        \begin{equation} \label{eq:mainbound}
            h^0(E) \leq \max\left\{\frac{1}{2}c_1(E)^2 + \frac{3}{2}c_1(E) + 1, r(E)\right\}. 
        \end{equation}
    Equivalently, if $k \geq \max\left\{\frac{1}{2}c_1(E)^2 + \frac{3}{2}c_1(E) + 1, r(E)\right\}$, then $B^k(\mathbf{v})$ is empty. 
    
    When $r(E) \geq \frac{1}{2}c_1(E)^2 + \frac{3}{2}c_1(E) + 1$, this bound is sharp, and $B^{r}(\vv)$ contains a component of the expected dimension. See Theorem \ref{pbound} and Theorem \ref{depth}. 
    \item \label{main1} (Nonemptiness) $B^r(\vv)$ is nonempty. See Theorem \ref{nonempty}.
    \item \label{main2} (Irreducibility) If $\ch_1(\vv) = 1$, all of the nonempty Brill-Noether loci on $M(\vv)$ are irreducible and of the expected dimension. See Theorem \ref{irred}. 
    \item \label{main3} (Reducibility) Suppose $\ch_1(\vv) > 1$. If $\mu(\vv) > 1/2$ is not an integer and $\Delta(\vv) \gg 0$, then $B^r(\vv)$ is reducible and contains components of both expected and unexpected dimensions. See Theorem \ref{comps}.
\end{enumerate}
\end{theorem}

Since the Brill-Noether loci are nested, Theorem \ref{main} (\ref{main1}) implies the Brill-Noether loci $B^k(\vv)$ are nonempty for all $0 \leq k \leq r$. The bound (\ref{eq:mainbound}) is also sharp for the line bundles $\OO(d)$ and the Lazarsfeld-Mukai-type bundles $M_d = \coker(\OO(-d) \rightarrow \OO^{h^0(\OO(d))})$, including $M_1 = T_{\PP^2}(-1)$.

Our main technical tools for studying Brill-Noether loci are parametrizations of certain moduli spaces $M(\vv)$ and Brill-Noether loci $B^k(\vv)$ by projective bundles whose fibers are (projectivizations of) extension spaces. The applications to Brill-Noether loci are often of the following form. A sheaf $E$ on $\PP^2$ has an \textit{evaluation map} on global sections
    \[\ev_E: H^0(\PP^2, E) \otimes \OO_{\PP^2} \rightarrow E.\]
Assume that the general $E \in M(\vv)$ has $m$ global sections, i.e., $h^0(\PP^2, E) = m$, and $m \leq r(E)$. When $\ev_E$ has full rank $E$ sits as an extension class
    \[[0 \rightarrow H^0(\PP^2, E) \otimes \OO_{\PP^2} \stackrel{\ev_E}{\rightarrow} E \rightarrow E' \rightarrow 0] \in \Ext^1(E', H^0(\PP^2, E') \otimes \OO_{\PP^2}).\]
If $E'$ is semistable with Chern character $\vv' \in K(\PP^2)$, we may form the projective bundle $\PP$ over $M(\vv')$ whose fiber over $E'$ is the projective space $\PP \Ext^1(E', H^0(\PP^2, E') \otimes \OO_{\PP^2})$. If the general such extension is stable, we obtain a classifying rational map $\phi: \PP \dashrightarrow M(\vv)$, which we call the \textit{extension parametrization} of $M(\vv)$.

The critical technical challenge in constructing such parametrizations is in proving the stability of extension sheaves. When the rank of $E'$ is 0, so $\chi(E) = r(E)$, we accomplish this by studying smooth curves in $\PP^2$ realized as the support of $E'$ and the moduli spaces $\Pic^d_{\mathcal{C}/U}$ of these sheaves (Sections \ref{boundssection} and \ref{redsection}). The central inputs to our results are the classification of stable Chern characters on $\PP^2$  by Dr\'ezet-Le Potier (\cite[Th\'eor\`eme C]{DLP}) and the computation of the cohomology of the general stable bundle in $M(\vv)$ by G\"ottsche-Hirschowitz (Theorem \ref{GH}). Our methods likely extend to other surfaces for which these two notions are understood, e.g., Hirzebruch surfaces and K3 surfaces of Picard rank 1 (see \cite{CH3} and \cite{CNY} respectively).

\subsection{Organization of the paper}
In Section \ref{preliminaries} we recall basic facts about moduli spaces of sheaves on $\PP^2$. In Section \ref{extparamsection} we introduce extension parametrizations, which are the main tools used in further sections.

In Section \ref{boundssection} we prove emptiness and non-emptiness statements for Brill-Noether loci. In Section \ref{redsection} we prove irreducibility and reducibility statements for Brill-Noether loci. 

\section{Preliminaries} \label{preliminaries}
\subsection{Stability on $\PP^2$} 
In this section we collect basic facts about stability for coherent sheaves on $\PP^2$. We refer the reader to the books by Le Potier \cite{LP} and Huybrechts-Lehn \cite{HL} for further details.

\subsubsection{Numerical invariants \& stability}
All sheaves in this paper will be coherent, but not necessarily torsion-free. Let $E$ be a coherent sheaf on $\PP^2$. The Hilbert polynomial $P_E(m) = \chi(E(m))$ is of the form
    \[P_E(m) = \alpha_d \frac{m^d}{d!} + O(m^{d-1}),\]
and we define the reduced Hilbert polynomial $p_E(m)$ to be
    \[p_E(m) = P_E(m)/\alpha_d.\]
A pure-dimensional coherent sheaf $E$ on $\PP^2$ is (semi)stable if for all nontrivial subsheaves $F \hookrightarrow E$ we have $p_F < (\leq) p_E$, where polynomials are compared for sufficiently large $m$. We will assume throughout this paper that a semistable sheaf has positive rank.

Given a character $\vv \in K(\PP^2)$, we define, respectively, the slope and the discriminant of $\vv$ by
    \[\mu(\vv) = \frac{\ch_1(\vv)}{\ch_0(\vv)}, \quad \Delta(\vv) = \frac{1}{2} \mu(\vv)^2 - \frac{\ch_2(\vv)}{\ch_0(\vv)}.\]
On $\PP^2$ these classes may be considered as rational numbers. When $\vv = \ch(E)$ for a coherent sheaf $E$ we set $\mu(E) = \mu(\vv)$ and $\Delta(E) = \Delta(\vv)$. We have
    \[\mu(E \otimes F) = \mu(E) + \mu(F), \quad \Delta(E \otimes F) = \Delta(E) + \Delta(F)\]
for coherent sheaves $E, F$ on $\PP^2$.

In terms of these invariants, the Riemann-Roch theorem takes the form
    \[\chi(E) = r(E)(p(\mu(E)) - \Delta(E)),\]
where
    \[p(x) = p_{\OO_{\PP^2}}(x) = \frac{1}{2}(x^2+3x+2).\]
More generally for sheaves $E,F$ on $\PP^2$, we set $\ext^i(E,F) = \dim \Ext^i(E,F)$. The relative Riemann-Roch theorem states
    \[\chi(E,F) = \sum (-1)^i \ext^i(E,F) = r(E)r(F)(p(\mu(E)-\mu(F)) - \Delta(E) - \Delta(F)),\]
and similarly for Chern characters.

Additionally, we may define slope-stability for sheaves on $\PP^2$: $E$ is slope-(semi)stable if for all subsheaves $F \hookrightarrow E$ of smaller rank, we have $\mu(F) < (\leq) \mu(E)$. Slope-stability implies stability, and is often easier to check.

\subsubsection{Classification of stable bundles}
The Chern characters of stable bundles on $\PP^2$ have been classified (see, e.g., \cite{DLP}, \cite{LP}). The classification is important for our results, and we sketch it in this subsection.

Recall that we say a character $\vv \in K(\PP^2)$ is stable if there is a stable sheaf of Chern character $\vv$, and that we assume stable characters have $\ch_0(\vv) > 0$. Stability imposes strong conditions on the numerical data of $\vv$. These conditions are determined by exceptional bundles on $\PP^2$. A sheaf $E$ on $\PP^2$ is \textit{exceptional} if $\Ext^i(E,E) = 0$ for $i > 0$.

By a theorem of Dr\'ezet \cite{Drezet-BSS}, every exceptional sheaf on $\PP^2$ is a stable vector bundle. A stable bundle $E$ is exceptional if and only if $\Delta(E) < 1/2$, by Riemann-Roch \cite[Proposition 16.1.1]{LP}. By definition, exceptional bundles are rigid and their moduli spaces consist of a single reduced point \cite[Corollary 16.1.5]{LP}. Exceptional bundles on $\PP^2$ include the line bundles $\OO_{\PP^2}(k)$ and the tangent bundle $T_{\PP^2}$, and all others can be formed from these by a process called mutation \cite{Drezet-BSS}. 

The slopes of exceptional bundles can be completely described \cite[\S 16.3]{LP}; set $\mathcal{E}$ to be the set of slopes of exceptional bundles. For an exceptional bundle $E$ of slope $\alpha$ we write $E = E_{\alpha}$ and $\Delta(E) = \Delta_{\alpha}$. We define the Dr\'ezet-Le Potier curve to be the following piecewise-polynomial curve in the $(\mu, \Delta)$-plane:
    \[\delta(\mu) = \sup_{\alpha \in \mathcal{E} : |\mu - \alpha| < 3} (p(-|\mu - \alpha|) - \Delta_{\alpha}).\]
A plot of an approximation of the Dr\'ezet-Le Potier curve between 0 and 1 is shown in Figure \ref{fig:DLP}. The precise relationship of exceptional bundles to stable characters is due to Dr\'ezet and Le Potier (\cite[Th\'eor\`eme C]{DLP}, \cite{LP}): a character $\vv \in K(\PP^2)$ is stable if and only if $c_1 := r\mu \in \Z$, $\chi := r(P(\mu) - \Delta) \in \Z$, and $\Delta(\vv) \geq \delta(\mu(\vv))$, or $\vv$ is exceptional. A character satisfying the first two conditions in the theorem is called integral. Each exceptional bundle $E_{\alpha}$ of slope $\alpha$ determines two ``branches'' of the Dr\'ezet-Le Potier curve, on the left and right sides of the vertical line $\mu = \alpha$. The characters $\vv$ on the branch to the right of $\mu = \alpha$ satisfy $\chi(\vv, E_{\alpha}) = 0$, and we call this branch the $^{\perp}E_{\alpha}$-branch of the curve (or left orthogonal branch), and characters $\vv$ on the branch to the left of $\mu = \alpha$ satisfy $\chi(E_{\alpha}, \vv) = 0$ and we call this the $E_{\alpha}^{\perp}$-branch (or right orthogonal branch).

\begin{figure}
    \centering
    \includegraphics{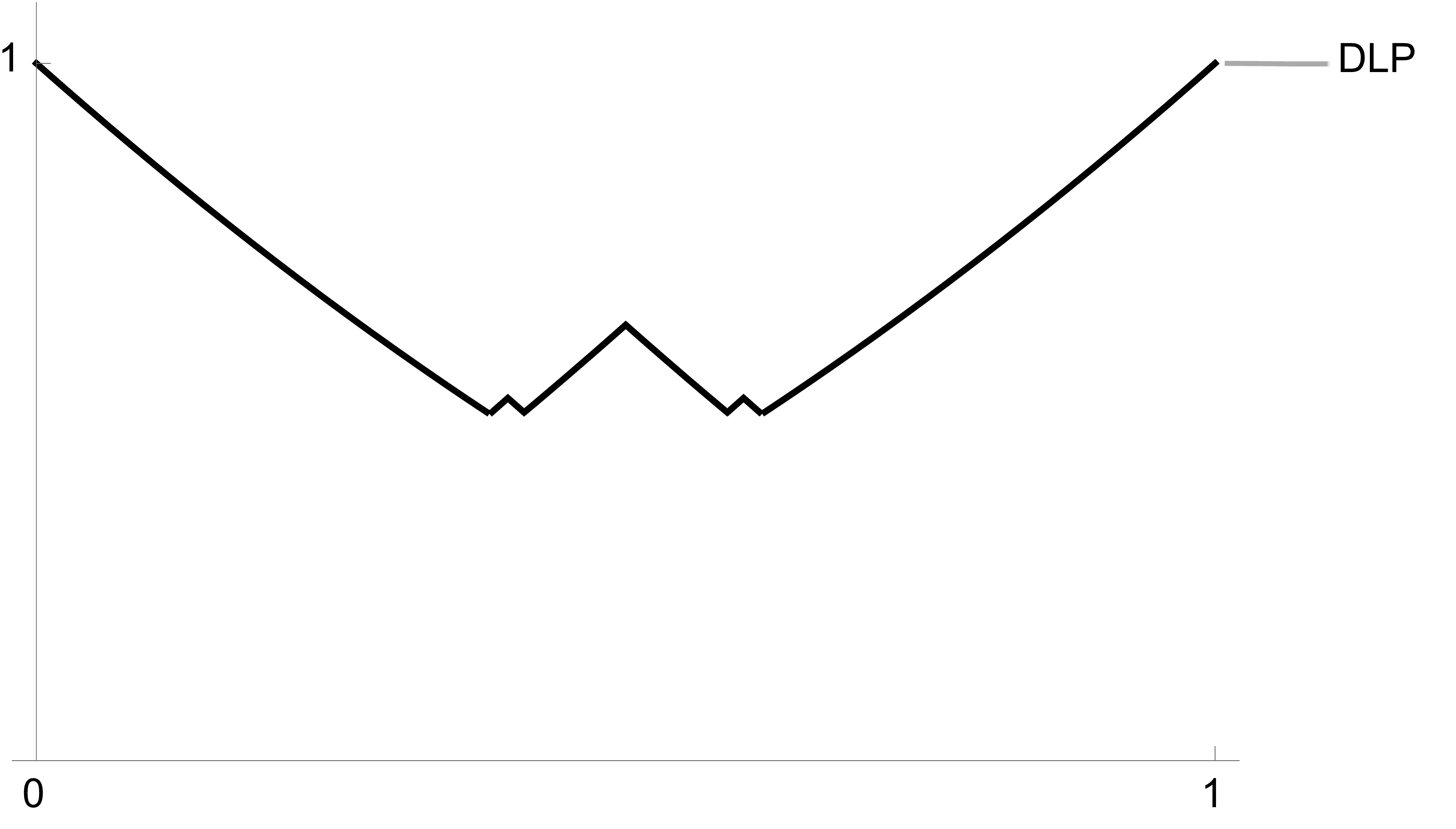}
    \caption{The Dr\'ezet-Le Potier curve $\Delta = \delta(\mu)$}
    \label{fig:DLP}
\end{figure}

Let $\vv \in K(\PP^2)$ be a character and set $\ch_0(\vv) = r$, $\mu(\vv) = \mu$, and $\Delta(\vv) = \Delta$. When $\vv$ is stable with $r>0$, Dr\'ezet and Le Potier showed the moduli space of stable sheaves with Chern character $\vv$ is a normal, irreducible, factorial projective variety of dimension
    \[\dim M(\vv) = r^2(2\Delta - 1) + 1\]
(\cite[Theorem 17.0.1]{LP}). When $\vv$ is a non-exceptional stable character of rank $> 1$, the general bundle $E \in M(\vv)$ is slope-stable (\cite[Corollaire 4.12]{DLP}). Slope-stability allows us to use elementary modifications:

\subsection{Elementary modifications} \label{jumpbyr}
If $\Delta_0$ is the discriminant of a stable non-exceptional sheaf $E$, then by integrality of the Euler characteristic one can see that the discriminant of another stable sheaf of slope $\mu(E)$ and rank $r(E)$ differs from $\Delta_0$ by an integral multiple of $\frac{1}{r(E)}$. One can in fact obtain any such discriminant larger than or equal to $\delta(\mu(E))$ by taking elementary modifications: choosing a point $p \in \PP^2$ and a surjection $E \rightarrow \OO_p$ one forms the exact sequence
    \[0 \rightarrow E' \rightarrow E \rightarrow \OO_p \rightarrow 0.\]
The sheaf $E'$ is not locally free, but it has $r(E') = r(E)$, $\mu(E') = \mu(E)$, and $\Delta(E') = \Delta(E) + 1/r(E)$. One can check that when $E$ is slope-stable, so is $E'$ (but this is not true for stability). In this way one can construct slope-stable bundles of discriminant $\Delta_0 + k/r(E)$ for any non-negative integer $k$. See \cite[Lemma 2.7]{CHBN} for details.

\subsection{Brill-Noether loci}
We now define the Brill-Noether loci and endow them with a natural determinantal structure, which leads to a lower bound on the dimension of their components.

The expected value of $h^0(E)$ when $E$ is stable is provided by the following well-known theorem due to G\"ottsche-Hirschowitz.

\begin{theorem}[\cite{GH}] \label{GH}
When $\ch_0(\vv) \geq 2$, the general sheaf $E \in M(\mathbf{v})$ has at most one nonzero cohomology group:
\begin{enumerate}
    \item if $\chi(E) \geq 0$ and $\mu(E) > -3$, then $h^1(E) = h^2(E) = 0$;
    \item if $\chi(E) \geq 0$ and $\mu(E) < -3$, then $h^0(E) = h^1(E) = 0$;
    \item if $\chi(E) < 0$, $h^0(E) = h^2(E) = 0$.
\end{enumerate}
\end{theorem}
    
When a sheaf $F \in M(\vv)$ has cohomology groups as prescribed in Theorem \ref{GH}, we will say that it is cohomologically general or has general cohomology. By semicontinuity, when $r \geq 2$ the locus of sheaves with general cohomology forms an open subset of the moduli space; the Brill-Noether loci $B^k(\vv)$ with $k > \chi(\vv)$ make up the complement of this open set. Let $M^s(\vv) \subseteq M(\vv)$ denote the locus of strictly stable sheaves.

\begin{definition}
For a stable Chern character $\mathbf{v} \in K(\PP^2)$, the \textit{$k$th Brill-Noether locus} $B^k(\mathbf{v}) \subseteq M(\mathbf{v})$ is defined as a set to be the closure
    \[B^k(\mathbf{v}) = \overline{\{E \in M^s(\mathbf{v}) : h^0(E) \geq k\}} \subseteq M(\mathbf{v}).\]
$B^k(\mathbf{v})$ is the locus where $h^0$ jumps by at least $k - \chi(\vv)$.
\end{definition}

The Brill-Noether loci are clearly nested: $B_k(\mathbf{v}) \supseteq B_{k+1}(\mathbf{v})$ for all $k$. 

\begin{remark}
We restrict to stable sheaves in the definition because the number of global sections is not constant in $S$-equivalence classes, so $h^0([E])$ is not defined for strictly semistable points $[E] \in M(\mathbf{v})$. 
\end{remark}

Note that for a stable sheaf $F$ with $\mu(F) < 0$, any map $\OO \rightarrow F$ corresponding to $0 \neq s \in H^0(F)$ would be destabilizing, so $H^0(F) = 0$ and there are no Brill-Noether loci in the associated moduli space. Thus in what follows we assume the slope is nonnegative. In these cases, we expect $h^0(F) = \max\{0, \chi(F)\}$. When $\mu < -3$, we have $h^0(E) = 0$ but $h^2(E) = h^0(E^{\vee}(-3))$ by Serre-duality, where $\mu(E^{\vee}(-3)) > 0$. Thus $h^2$-jumping loci are Serre-dual to $h^0$-jumping loci when the slope is sufficiently negative. When $-3 < \mu(E) < 0$, Serre duality implies that $h^0(E) = h^2(E) = 0$, so there are no Brill-Noether loci for any cohomology.

We endow $B^k(\mathbf{v})$ with a determinantal scheme structure, as follows (\textit{cf.} \cite[Proposition 2.6]{CHW} and \cite[\S 2]{CMR}). Let $\mathscr{E}/S$ be a proper flat family of stable sheaves on $\PP^2$ of Chern character $\mathbf{v}$. We will examine the relative Brill-Noether locus $B^k_S(\mathbf{v})$; when $M^s(\mathbf{v})$ admits a universal family $\mathscr{U}$, this will give the appropriate scheme structure on $B^k(\mathbf{v})$. In general, one can work on the stack $\mathcal{M}^s(\mathbf{v})$ and take the image in the coarse moduli space $M^s(\mathbf{v})$ (see \cite[Theorem 4.16 and Example 8.7]{A}). 

Choose $s \in S$ and let $d \gg 0$ be sufficiently large and let $C \subseteq \PP^2$ be a general-enough curve of degree $d$, chosen so that the singularities of $\mathscr{E}_s$ do not meet $C$. After replacing $S$ by an open subset, we can assume $C$ does not meet the singularities of any member of $\mathscr{E}$. Let $p: S \times \PP^2 \rightarrow S$ and $q: S \times \PP^2 \rightarrow \PP^2$ be the projections. Form the exact sequence
    \[0 \rightarrow \mathscr{E} \rightarrow \mathscr{E} \otimes q^*\OO(d) \rightarrow \mathscr{E} \otimes (q^*\OO(d)|_C) \rightarrow 0\]
of sheaves on $S \times \PP^2$. 

Choosing $d$ large enough so that each of the cohomology groups $H^1(\mathscr{E}_s \otimes q^*\OO(d))$ vanish, we conclude $R^1p_*(\mathscr{E} \otimes q^*\OO(d)) = 0$ and we obtain the exact sequence
    \[0 \rightarrow p_*\mathscr{E} \rightarrow p_*(\mathscr{E} \otimes q^*\OO(d)) \stackrel{\phi}{\rightarrow} p_*(\mathscr{E} \otimes (q^*\OO(d)|_C)) \rightarrow R^1p_*(\mathscr{E}) \rightarrow 0\]
of sheaves on $S$. When $d$ is sufficiently large the map $\phi$ is a map between bundles of the same rank, and the $k$th determinantal variety associated to $\phi$ is supported on those $s \in S$ such that $h^0(\mathscr{E}_s) \geq k$, or equivalently $h^1(\mathscr{E}_s) \geq k - \chi(\vv)$. Thus scheme-theoretically we define $B^k(\mathbf{v})$ to be this determinantal variety, and observe that its support agrees with the set described above.

The construction above does not depend on the choice of $d$ or $C$; precisely, over an open subset $U$ of $S$ over which the source and target of $\phi$ are trivial, the Fitting ideal $\text{Fitt}_0(R^1p_*\mathcal{E})$ is generated by $\det \phi$, viewed as a section of $\OO_U$. See \cite[Chapter 20]{eisenbud}. The Fitting ideal is the ideal sheaf of $B^k(\vv)$, and gives the desired scheme structure.

Being a determinantal variety, the codimension of a component $Z$ of $B^k(\mathbf{v})$ is bounded:
    \[\text{codim}_{M(\mathbf{v})}(Z) \leq k(k-\chi(\vv)).\]
Thus since, for a general bundle $E \in M(\mathbf{v})$, we have
    \[\dim(M(\mathbf{v})) = \ext^1(E,E) = r(\mathbf{v})^2(2\Delta(\mathbf{v})-1)+1\]
the dimension of a component $Z$ of $B^k(\mathbf{v})$ is bounded below:
    \[\dim(Z) \geq r(\mathbf{v})^2(2\Delta(\mathbf{v})-1)+1 - k(k-\chi(\vv)).\]

\begin{definition}
The \textit{expected codimension} of the $k$th Brill-Noether locus $B^k(\vv) \subseteq M(\vv)$ is
    \[\text{expcodim}^k(\vv) = k(k-\chi(\vv)).\]
\end{definition}

\section{Extension parametrizations of moduli spaces} \label{extparamsection}
In this section we study the stability of extension sheaves. When extension sheaves are semistable, we obtain families of sheaves over loci in the associated moduli space.

\begin{definition}
Suppose that $\vv, \vv', \vv'' \in K(\PP^2)$ are stable characters with $\vv=\vv'+\vv''$. Let $\mathscr{E}' /S'$ and $\mathscr{E}'' /S''$ be families of semistable sheaves of Chern characters $\vv'$ and $\vv''$, respectively. If the general extension $E$
    \[0 \rightarrow \mathscr{E}'_{s'} \rightarrow E \rightarrow \mathscr{E}''_{s''} \rightarrow 0\]
with $s' \in S'$ and $s'' \in S''$ is semi-stable, then an induced rational map
    \begin{center}
    \begin{tikzcd}
        \PP \Ext^1(\mathscr{E}''_{s''}, \mathscr{E}'_{s'}) \arrow[d] \arrow[r, dashed] & M(\vv) \\
        S' \times S''
    \end{tikzcd}
    \end{center}
defined on a locus where $\Ext^1(\mathscr{E}''_{s''}, \mathscr{E}'_{s'})$ has constant rank, is called an extension parametrization associated to $M(\vv)$. 
\end{definition}

Stability of these extensions is the central technical challenge of this paper. However, there is a class called the \textit{extremal} extensions whose stability is not hard to show, which we now recall.

\subsection{Extremal extensions} \label{extremalsection}
Stability of extension sheaves is known when the slope of the subsheaf is \textit{extremal} with respect to the extension sheaf. The following definition is derived from \cite[Definition 4.1]{CH1}, but we make some slight modifications for our purposes.

\begin{definition}\label{extremaldef}
A Chern character $\vv'$ on $\PP^2$ is called \textit{extremal} for $\vv$ if $\vv'$ satisfies:
    \begin{enumerate}
    \item[(D1) \label{D1}] $\ch_0(\vv') < \ch_0(\vv)$; 
    \item[(D2) \label{D2}] $\mu(\vv')\leq \mu(\vv)$; 
    \item[(D3) \label{D3}] $\vv'$ is stable;
    \end{enumerate}
and furthermore, for any $\ww \in K(\PP^2)$ satisfying (\hyperref[D1]{D1})-(\hyperref[D3]{D3}):
    \begin{enumerate}
    \item[(E1) \label{E1}] $\mu(\mathbf{w})\leq \mu(\vv')$;
    \item[(E2) \label{E2}] If $\mu(\ww)=\mu(\vv')$, then $\Delta(\mathbf{w}) \geq \Delta(\vv')$;
    \item[(E3) \label{E3}] If $\mu(\ww)=\mu(\vv')$ and $\Delta(\ww)=\Delta(\vv')$, then $\ch_0(\vv')\leq \ch_0(\mathbf{w})$.
    \end{enumerate}
A triple $\Xi = (\mathbf{v}', \mathbf{v}, \mathbf{v}'')$ of Chern characters on $\PP^2$ is called \textit{extremal} if $\vv' + \vv'' = \vv$, where $\vv'$ is the extremal character for $\vv$, and $\vv''$ is stable.
\end{definition}

It is not obvious \textit{a priori} that extremal decompositions exist for any $\vv \in K(\PP^2)$ since $\vv''$ may not be stable, but at least it is clear when $\vv$ has large discriminant, see e.g. \cite[Lemma 4.3]{CH1}, or simply observe that $\Delta(\vv'')\gg 0$ when $\Delta(\vv)\gg 0$.

Our main use of Definition \ref{extremaldef} will be to construct extension parametrizations associated to moduli spaces. The following simple observation is crucial.

\begin{proposition} \label{extremal}
Suppose $\Xi = (\vv', \vv, \vv'')$ is an extremal triple, and we have a nonsplit short exact sequence
    \[0 \rightarrow E' \rightarrow E \rightarrow E'' \rightarrow 0\]
with $E' \in M(\vv')$ and $E'' \in M(\vv'')$. Then $E$ is semistable. 
\end{proposition}

\begin{proof}
Suppose that $E \rightarrow Q$ is a destabilizing quotient of $E$, so that $\mu(Q) \leq \mu(E)$ with $r(Q) < r(E)$. After perhaps passing to a further quotient of $Q$, we may assume $Q$ is stable. Slope-closeness (\hyperref[E1]{E1}) implies that $\mu(Q) \leq \mu(E')$. If $\mu(Q) < \mu(E')$, the composition $E' \rightarrow E \rightarrow Q$ is zero by stability. So there is an induced map $E'' \rightarrow Q$. However $\mu(\vv'') \geq \mu(\vv) > \mu(Q)$, hence the induced map is zero. This implies that $E \rightarrow Q$ vanishes. 

We now have $\mu(Q) = \mu(E')$. Discriminant-minimality (\hyperref[E2]{E2}) implies that $\Delta(Q) \geq \Delta(E')$. If $\Delta(Q) > \Delta(E')$, we get a contradiction as in the preceding paragraph. 

Hence $\mu(Q) = \mu(E')$ and $\Delta(Q) = \Delta(E')$. By (\hyperref[E3]{E3}), $E', Q$ are stable. Since $E$ is a non-split extension, $E' \rightarrow Q$ is zero, inducing a map $E'' \rightarrow Q$. Now by (\hyperref[D1]{D1}), (\hyperref[D2]{D2}) this map can be non-zero only if $\mu(\vv'')=\mu(\vv)=\mu(\vv')$ and $\Delta(\vv'')=\Delta(\vv)=\Delta(\vv')$. Hence $E$ is semi-stable.
\end{proof}

\begin{corollary}
Let $\vv \in K(\PP^2)$ be a stable character which admits an extremal decomposition $(\vv', \vv, \vv'')$, and let $\mathscr{E}'/S'$ and $\mathscr{E}''/S''$ be families of semistable sheaves of characters $\vv'$ and $\vv''$ respectively. Then the induced extension parametrization
    \begin{center}
    \begin{tikzcd}
        \PP \mathrm{Ext}^1(\mathscr{E}''_{s''}, \mathscr{E}'_{s'}) \arrow[r, dashed] \arrow[d] & M(\vv) \\
        S' \times S''
    \end{tikzcd}
    \end{center}
defined over the locus where $\mathrm{Ext}^1(\mathscr{E}''_{s''}, \mathscr{E}'_{s'})$ has minimal rank, exists, where $s' \in S'$ and $s'' \in S''$.
\end{corollary}

\begin{proof}
By Proposition \ref{extremal}, for any $s' \in S'$ and $s'' \in S''$, a nonsplit extension
    \[0 \rightarrow \mathscr{E}'_{s'} \rightarrow E \rightarrow \mathscr{E}''_{s''} \rightarrow 0\]
is semistable of Chern character $\vv$. So on any locus in $S' \times S''$ where $\Ext^1(\mathscr{E}''_{s''}, \mathscr{E}'_{s'})$ is of constant rank, there is an induced map $\PP \dashrightarrow M(\vv)$, where $\PP$ is the projective bundle over $S' \times S''$ with fiber $\PP \Ext^1(\mathscr{E}''_{s''}, \mathscr{E}'_{s'})$ over $(s', s'') \in S' \times S''$.
\end{proof}

\begin{remark}
	When the subsheaf is not extremal for the extension sheaf, in general additional conditions are required to prove generic stability, see for instance Lemma \ref{corank0lemma} and Lemma \ref{prioritary}.
\end{remark}

\section{Emptiness \& non-emptiness of Brill-Noether loci} \label{boundssection}
For a stable Chern character $\vv \in K(\PP^2)$ there are only finitely many values of $k \geq 0$ such that $B^k(\vv) \subseteq M(\vv)$ is nonempty. That is, there are only finitely many values of $h^0(E)$ for a stable sheaf $E \in M(\vv)$. We first give a bound on this quantity. Later on in this section we will analyze which such values are achieved.

\begin{theorem} \label{pbound}
If $E$ is a slope-semistable sheaf on $\mathbb{P}^2$ with Chern character $\mathbf{v}$ such that $\ch_1(\vv) \geq 0$, then
    \[h^0(E) \leq \max\{p(\ch_1(\vv)), \ch_0(\vv)\},\]
where $p(x) = p_{\OO_{\PP^2}}(x) = \frac{x^{2}+3x+2}{2}$.
\end{theorem}

In the proof we will consider partial evaluation maps $V \otimes \OO \rightarrow E$ for a subspace $V \subseteq H^0(E)$. We fix $E$ and $V$ for the following lemma.

\begin{lemma} \label{pbound-lemma}
	Let $F \subseteq E$ be the image sheaf of the evaluation map $V \otimes \mathcal{O} \rightarrow E$, and let $0 = F_0 \subseteq F_1 \subseteq \cdots \subseteq F_m = F$ be the Harder-Narasimhan filtration of $F$, with graded pieces $\gr_i = F_i/F_{i-1}$.
	
	If $r(F) < r(E)$, then $c_1(\gr_i) < c_1(E)$ for each $i$.
\end{lemma}

\begin{proof}
	If not, then since $r(\gr_i) \leq r(F) < r(E)$,
    	\[\mu(E) = \frac{c_1(E)}{r(E)} \leq \frac{c_1(\gr_i)}{r(E)} < \frac{c_1(\gr_i)}{r(\gr_i)} = \mu(\gr_i).\]
	(Note that we need $c_1(E) \geq 0$ for $V$ to be nonempty.) The inclusions $\gr_1 = F_1 \hookrightarrow F \hookrightarrow E$ and semistability of $E$ imply
    	\[\mu(\gr_1) \leq \mu(E) < \mu(\gr_i)\]
	contradicting that the slopes of the $\gr_i$ are non-increasing. We conclude that $c_1(\gr_i) < c_1(E)$ for each $i$. 
\end{proof}

We observe for use in the cases below that the function $p(x)/x$ is increasing for integers $x \geq 2$, with $p(1) = p(2)/2$.

\begin{proof}[Proof of Theorem \ref{pbound}]
First, since we have $E \subseteq E^{\vee \vee}$, we have $h^0(E) \leq h^0(E^{\vee \vee})$. Because $E^{\vee \vee}$ is slope-semistable with $r(E^{\vee \vee}) = r(E)$ and $c_1(E^{\vee \vee}) = c_1(E)$, it suffices to prove the theorem when $E$ is locally free.

\textit{Case I.} In this case we assume $r(E) = \ch_0(\vv) \geq p(\ch_1(\vv))$, and proceed by induction on $r(E)$. In this case when $r(E) = 1$ we have $c_1(E) \leq 0$, and when $c_1(E) < 0$ slope-semistability implies $h^0(E) = 0$, as desired. When $c_1(E) = 0$, it follows that $E$ is the ideal sheaf of a finite length subscheme $Z \subseteq \PP^2$. Thus $h^0(E) \leq 1$, and equality holds if and only if $Z$ is empty. 

We first show that any $r(E)$-dimensional subspace $V \subseteq H^0(E)$ spans a full rank subsheaf of $E$. That is, the evaluation map $V \otimes \OO \rightarrow E$ is full rank. Suppose not, so that $r(F) < r(E)$. To show this we first claim that each $\gr_i$ has $r(\gr_i) \geq p(c_1(\gr_i))$, so by induction on the rank we have $h^0(\gr_i) \leq r(\gr_i)$. It follows that
    \[h^0(F) \leq \sum h^0(\gr_i) \leq \sum r(\gr_i) = r(F) < r(E).\]
Since by definition $h^0(F) \geq r(E)$, we conclude from this contradiction that any partial evaluation map of any $r(E)$ independent sections in this case is of full rank.

To check the claim, assume otherwise that $r(\gr_i) < p(c_1(\gr_i))$. It is enough to produce a contradiction for those $\gr_i$ with $c_1(\gr_i) > 0$, since when $c_1(\gr_i) \leq 0$ we have $h^0(\gr_i) \leq r(\gr_i)$. Then
    \[\mu(E) = \frac{c_1(E)}{r(E)} \leq \frac{c_1(E)}{p(c_1(E))} \leq \frac{c_1(\gr_i)}{p(c_1(\gr_i))} < \frac{c_1(\gr_i)}{r(\gr_i)} =\mu(\gr_i),\]
where we have used Lemma \ref{pbound-lemma} and that $x/p(x)$ is non-increasing when $x \geq 1$. The slopes $\mu(\gr_i)$ in the Harder-Narasimhan filtration are decreasing and $E$ is semistable, so we have
    \[\mu(\gr_i) \leq \mu(\gr_1) \leq \mu(E)\]
so this contradiction proves the claim. We conclude that any $r(E)$ sections span a full rank subsheaf of $E$.

Therefore an $r(E)$-dimensional subspace $V \subseteq H^0(E)$ determines a nonzero section of $\OO(c_1(E))$ via the natural map
    \[\wedge^{r(E)} H^0(E) \rightarrow H^0(\wedge^{r(E)} E) = H^0(\OO(c_1(E))),\]
where we have $\wedge^r E = \OO(c_1(E))$ because $E$ is locally free. In this way we obtain a morphism
    \[f: G(r(E), H^0(E)) \rightarrow \PP H^0(\OO(c_1(E))).\]

When $h^0(E) \geq r(E) + 1$, it is easy to see that $f$ is nonconsant. Indeed, any $r(E)$-dimensional subspace $V \subseteq H^0(E)$ sits in a short exact sequence
    \[0 \rightarrow V \otimes \OO \rightarrow E \rightarrow Q \rightarrow 0\]
because the evaluation map is of full rank. Choose $s \in H^0(E) \smallsetminus V$, and consider a general point $p \in \Supp(Q)$. Then $0 \neq s_p \in H^0(Q|_p)$ does not lie in $V|_p \subseteq E|_p$. Choosing a generating set $s_1, ..., s_{r(E)}$ for the span $\langle s_p, V|_p \rangle \subseteq E_p$; then the $s_i$ span a rank $r(E)$ subspace $V' \subseteq H^0(E)$ such that the induced section $\wedge^{r(E)}V' \in H^0(\wedge^{r(E)} E)$ is independent from $\wedge^{r(E)} V$. For instance, the multiplicity of $\wedge^{r(E)} V'$ at $p$ is lower than that of $\wedge^{r(E)}V$.

However if $h^0(E) \geq r(E)+1$, we have 
    \[\dim G(r(E),H^0(E)) \geq r(E) \geq p(c_1(E))\]
since $E$ has $r(E) \geq p(c_1(E))$, and $\dim \PP H^0(\OO(c_1(E))) = p(c_1(E)) - 1$. Since the Picard rank of a Grassmannian is 1, there is no non-constant map from a Grassmannian to a lower-dimensional projective variety. We conclude that $h^0(E) \leq r(E)$, as desired. 

\textit{Case II.} We are now in the situation that $r(E) < p(c_1(E)) = \frac{c_1(E)^2 + 3c_1(E) + 2}{2}$, and we want to prove $h^0(E) \leq p(c_1(E))$. We will use induction on $c_1(E)$. If $c_1(E) = 0$, then $p(c_1(E))=1$ and there is nothing to prove.  

We first show that any $p(c_1(E))$-dimensional subspace $V \subseteq H^0(E)$ spans a full rank subsheaf of $E$. We proceed by contradiction, so we let $F$ be the image of the evaluation map $V \otimes \mathcal{O} \rightarrow E$ and assume $r(F) < r(E)$. Since $p(x)/x$ is non-decreasing, when $c_1(\gr_i) \geq 1$ we have
    \[p(c_1(\gr_i)) = c_1(\gr_i) \frac{p(c_1(\gr_i))}{c_1(\gr_i)} \leq c_1(\gr_i) \frac{p(c_1(E))}{c_1(E)} \leq p(c_1(E)) \frac{r(\gr_i)}{r(E)}\]
by Lemma \ref{pbound-lemma}. The $\gr_i$ may fall into either Case I or Case II. When $\gr_i$ is in Case II, we obtain
    \[h^0(\gr_i) \leq p(c_1(\gr_i)) \leq p(c_1(E)) \frac{r(\gr_i)}{r(E)}\]
by induction on $c_1$. When $c_1(\gr_i) < 1$, we have
    \[h^0(\gr_i) \leq r(\gr_i) = r(E) \frac{r(\gr_i)}{r(E)} \leq p(c_1(E)) \frac{r(\gr_i)}{r(E)}\]
by the assumption that $E$ is in Case II. For those $\gr_i$ in Case I with $r(\gr_i) \geq p(c_1(\gr_i))$, we also have
    \[h^0(\gr_i) \leq r(\gr_i) = r(E) \frac{r(\gr_i)}{r(E)} < p(c_1(E)) \frac{r(\gr_i)}{r(E)}.\]
We now have
    \begin{equation} \label{eq:F}
        h^0(F) \leq \sum h^0(\gr_i) \leq \sum p(c_1(E)) \frac{r(\gr_i)}{r(E)} = p(c_1(E)) \frac{r(F)}{r(E)} < p(c_1(E)). 
    \end{equation}
By definition we have $h^0(F) \geq p(c_1(E))$, and this contradiction proves that any $p(c_1(E))$-dimensional subspace $V$ spans a full rank subsheaf of $E$.

We can additionally assume that the dependency locus of $V$ is zero-dimensional. If not, consider the image $F \subseteq E$ of the evaluation map $V \otimes \OO \rightarrow E$ and its Harder-Narasimhan filtration. Then $c_1(F) < c_1(E)$. By (\ref{eq:F}), we know that $h^0(F) \leq p(c_1(E))$ (only the final inequality in (\ref{eq:F}) is no longer strict), and equality holds only if for every $i$, $h^0(\gr_i)= p(c_1(E))\frac{r(\gr_i)}{r(E)}$ for every graded piece. This is true only if $c_1(\gr_i)>0$; $r(\gr_i)< p(c_1(\gr_i))$; $c_1(\gr_i)=1$; and $c_1(E)=2$. In this case, $c_1(F)<c_1(E)=2$, $c_1(F)\leq 1$. If $F$ is unstable, then there is at least one Harder-Narasimhan factor $\gr_i$ whose degree is non-positive, which is not the case. If $F$ is stable, it satisfies $r(F)<p(c_1(F))$, then by induction, $h^0(F)\leq p(c_1(F))<p(c_1(E))$. In any case, we have $h^0(F)<p(c_1(E))$, a contradiction.

We now return to the bound on $h^0(E)$. We first prove $h^0(E) \leq p(c_1(E))$ when $0 \leq \mu(E) < 1$. We have just shown that any $p(c_1(E))$ sections of $E$ span a full rank subsheaf, and their dependency locus is zero-dimensional. Let $L \subseteq \PP^2$ be a general line. Consider 
    \[0 \rightarrow E(-1) \rightarrow E \rightarrow E|_L \rightarrow 0\]
and write $E|_L = \bigoplus_{i=1}^{r(E)} \OO_L(d_i)$. Since $\mu(E) < 1$, there is at least one $d_i \leq 0$, say $d_1$. Taking $H^0$ we obtain $\pi: H^0(E) \rightarrow \bigoplus_{i=1}^{r(E)} H^0(\OO_L(d_i))$. Let $W \subseteq \bigoplus_{i=1}^{r(E)} H^0(\OO_L(d_i))$ be the image of $\pi$. If $W \subseteq \bigoplus_{i=2}^{r(E)} H^0(\OO_L(d_i))$, then sections of $E$ are dependent along $L$, hence $h^0(E) < p(c_1)$. If $W \not\subseteq \bigoplus_{i=2}^{r(E)} H^0(\OO_L(d_i))$, then $d_1 = 0$, and since $\pi^{-1}(W \cap \bigoplus_{i=2}^{r(E)} H^0(\OO_L(d_i)))$ has codimension 1, there are $h^0(E)-1$ sections dependent along $L$, hence $h^0(E)-1 < p(c_1)$. That is, $h^0(E) \leq p(c_1(E))$, as desired.

Now if $n \leq \mu < n+1 $, we induct on $n$. We have just shown the statement is true for $n=0$; now assume $n \geq 1$. Again we use that any $p(c_1(E))$ sections span a full rank subsheaf, and the dependency locus 
is supported in dimension zero. Consider again the restriction to a line:
    \[0 \rightarrow E(-1) \rightarrow E \rightarrow E|_L \rightarrow 0.\]
Taking $ H^0$ we get $h^0(E) \leq h^0(E(-1))+ h^0(E|_L)$. By induction, $h^0(E(-1)) \leq p(c_1(E) - r(E))$. Consider $E|_L = \bigoplus_{i=1}^{r(E)} \OO_L(d_i)$. If some $d_i < 0$, say $d_1$, then the span of sections in $W$ is contained in $\bigoplus_{i \geq 2} \OO_L(d_i)$, and it follows that $h^0(E) < p(c_1(E))$. Hence we assume $d_i \geq 0$ for all $i$. Then $h^0(E|_L) = \sum_{i=1}^{r(E)} h^0(\OO_L(d_i+1)) = c_1(E) + r(E)$. Now since $n \geq 1$, $c_1(E) \geq r(E) \geq 1$ and we obtain
    \begin{align*}
        r(E)(2c_1(E) - r(E) + 1) - 2c_1(E) &= (r(E)-1)2c_1(E) - r(E)^2 + r(E) \\
        &\geq (r(E)-1)2r(E) - r(E)^2 + r(E) \\
        &= r(E)(r(E) - 1) \geq 0
    \end{align*}
hence $r(E)(2c_1(E) - r(E) + 1) \geq 2c_1(E)$, and 
    \begin{align*}
        p(c_1(E)) - p(c_1(E)-r(E)) &= \frac{2c_1(E)r(E) - r(E)^2 + 3r(E)}{2} \\
        &= r(E) + \frac{r(E)(2c_1(E) - r(E) + 1)}{2} \\
        &\geq r(E) + c_1(E).
    \end{align*}
In particular we have 
    \[h^0(E) \leq h^0(E(-1)) + h^0(E|_L) \leq p(c_1(E) - r(E))+(r(E) + c_1(E)) \leq p(c_1(E)).\] 
This finishes the proof.
\end{proof}

We will later need the following corollary, which is an immediate consequence of Theorem \ref{pbound}.

\begin{corollary} \label{degone} 
If $E$ is a semistable sheaf on $\PP^2$ with $c_1(E) = 1$ and $r(E) \geq 3$, then $h^0(E) \leq r(E)$.
\end{corollary}

\begin{remark} \label{rmkdegone}
We can describe the behavior of $E$ with $c_1(E) = 1$ also when the rank is small. If the rank is 1, then $E$ is a twist of an ideal sheaf $E \simeq I_Z(1)$. Then $h^0(E) \neq 0$ if and only if $Z$ is colinear, in which case
    \begin{itemize}
        \item $h^0(E) = 1$ if $\ell(Z) \geq 2$,
        \item $h^0(E) = 2$ if $\ell(Z) = 1$, and
        \item $h^0(E) = 3$ if $\ell(Z) = 0$. 
    \end{itemize}

If the rank is 2, suppose $h^0(E) \geq 3$ and consider the partial evaluation map $V \otimes \OO \rightarrow E$ for $V \subseteq H^0(E)$ a rank 3 subspace. Let $F$ be the image of $V \otimes \mathcal{O}$ in $E$. If the rank of $F$ is one, then it falls into one of the above cases. Since $r(V) = 3$, we have $h^0(F) \geq 3$, so as above $F = \mathcal{O}$, but in this case the kernel $K = \ker(V \otimes \mathcal{O} \rightarrow F)$ has $h^0(K) \neq 0$, which contradicts the definition of $F$. It follows that $r(F) = 2$, and in fact is semistable of slope $1/2$.

This produces a short exact sequence
    \[0 \rightarrow K \rightarrow V \otimes \OO \rightarrow F \rightarrow 0\]
with $r(K) = 1$, $c_1(K) = -1$. It follows that $K \simeq I_Z(-1)$ for some codimension two subscheme $Z \subseteq \PP^2$. In particular, $\ch_2(K) = \frac{1}{2} - \ell(Z)$. Thus
    \[\Delta(F) = \frac1{2}\left( \frac1{2} \right)^2 - \frac{\ell(Z) - 1/2}{2} = \frac{3}{8} - \frac{\ell(Z)}{2}.\]
Semistability additionally implies that $\Delta(F) = \frac{3}{8} + \frac{k}{2}$ for some $k \geq 0$ (see Section \ref{jumpbyr}), from which we find $k = -\ell(Z) \leq 0$, so $k = 0$ and $F \simeq E = T_{\PP^2}(-1)$. In all other cases, we again have $h^0(E) \leq r(E)$.
\end{remark}


\subsection{The Brill-Noether loci $B^r(\vv)$} In this section we show that $B^k(\vv) \neq \emptyset$ whenever $k \leq \ch_0(\vv)$. We can say a great deal about the loci of sheaves $E$ with $h^0(E) \geq r = r(E)$, i.e., the Brill-Noether loci $B^r(\vv)$. We will need the following statements to study the Brill-Noether loci when $\ch_0(\vv) \geq p(\ch_1(\vv))$.

\begin{lemma} \label{corank0lemma}
Let $\vv \in K(\PP^2)$ be a stable Chern character and write $r = \ch_0(\vv)$ and $c_1 = \ch_1(\vv)$. Assume that $r \geq p(\ch_1(\vv))$. 
\begin{enumerate}
    \item \label{one} For every $E \in B^r(\vv)$, evaluation map on global sections $\ev_E$ sits in an exact sequence
        \[0 \rightarrow \OO^r \rightarrow E \rightarrow Q_E \rightarrow 0,\]
    where $Q_E$ is a pure sheaf supported on a 1-dimensional subscheme of $\PP^2$.
    \item \label{two} Let $E \in B^r(\vv)$ be a general sheaf. If $Q_E$ is supported on a smooth curve $C$ of degree $c_1(E)$, then $Q_E$ is the pushforward of a line bundle on $C$. Writing $Q_E \simeq \OO_C(D)$, we have 
        \[\deg_C D \leq \frac{c_1^2+3 c_1}{2} - r \leq -1.\]
    \item \label{three} Conversely, let $\OO_C(D)$ be the pushforward of a line bundle on a smooth curve $C \subseteq \PP^2$, where $\deg_C D \leq \frac{1}{2}(\deg(C)^2+3\deg(C))-r \leq -1 $. Then a general extension sheaf
        \[0 \rightarrow \OO^r \rightarrow E \rightarrow \OO_C(D) \rightarrow 0\] 
    is stable.
\end{enumerate}
\end{lemma}

\begin{proof}
(\ref{one}) First, the proof of Theorem \ref{pbound}, Case I, implies that the evaluation map has full rank. It follows that $Q_E$ is a torsion sheaf. Consider a negative twist
    \[0 \rightarrow \OO(-m)^r \rightarrow E(-m) \rightarrow Q_E(-m) \rightarrow 0\] 
of the above exact sequence, where $m \gg 0$. If $Q_E$ is has torsion $T \subseteq Q_E$ supported in codimension two, then $T(-m) \simeq T$, so in particular $T(-m)$ and $Q_E(-m)$ have global sections. The associated cohomology exact sequence 
    \[H^0(E(-m)) \rightarrow H^0(Q_E(-m)) \rightarrow H^1(\OO(-m)^r)\]
has outer terms which vanish and an inner term which does not. This contradiction proves that $Q_E$ is supported in pure codimension one.
    
(\ref{two}) There is a filtration $ 0 = F_k Q_E \subseteq F_{k-1} Q_E \subseteq \cdots \subseteq F_1 Q_E \subseteq Q_E$, where each graded piece is the pushforward of a sheaf on $C$ \cite{Drezet-mult}. If $F_i Q_E \neq 0$ then its support is $C$, otherwise it would be a torsion subsheaf of $Q_E$ supported on zero dimensional scheme, contradicting the purity of $Q_E$ from (\ref{one}). Now if $F_1 Q_E \neq 0$, then $c_1(F_1 Q_E) \geq c_1$, but $Q_E/F_1 Q_E$ has support $C$, whose first Chern class is $ \geq c_1$, hence $c_1(Q_E) \geq 2c_1$, a contradiction.

For the degree, note that since $r \geq p(c_1)$, we have
    \[\mu(E) \leq \frac{c_1}{p(c_1)} \leq 1/3\]
since $x/p(x) \leq 1/3$ for integers $x$. Thus $E$ lies directly above the $^{\perp}\OO$-branch of the Dr\'ezet-Le Potier curve, which gives $\chi(E, \OO) \leq 0$. Thus
    \[\chi (\OO_C(D), \OO_{\PP^2}) = \chi(E, \OO_{\PP^2} )-\chi(\OO_{\PP^2}^r, \OO_{\PP^2}) \leq -r.\]

On the other hand,
    \[\chi(\OO_C(D),\OO_{\PP^2}) = \chi(\OO_C, \OO_C(D-3H)) = \deg_C(D-3H) + 1 - g(C)\]
so
    \[\deg_C D-3c_1 + 1 - \frac{c_1^2 - 3c_1 + 2}{2} \leq -r \]
i.e., $\deg_C D \leq \frac{1}{2}(c_1^2+3c_1)-r$. The last inequality follows from the assumption $r \geq p(c_1)$.
    
(\ref{three}) Let $F$ be a possible destabilizing semistable quotient sheaf of $E$. Since $F$ is torsion-free, there is no nonzero map
from $\OO_C(D) \rightarrow F$, hence the composition $\OO^r \rightarrow E \rightarrow F$ is nonzero. Since $\mathcal{O}^r \rightarrow E$ is of full rank, the composition $\OO^r \rightarrow E \rightarrow F$ is of full rank as well. The sheaves $\OO^r$ and $F$ are both semistable, which gives $ \mu(F) \geq 0$. On the other hand, since $F$ has smaller rank than $E$, $c_1(F)< c_1(E)$, hence the assumption on $\deg(D)$ implies $F$ also has $r(F) \geq p(c_1(F))$. Therefore $ h^0(F) \leq r(F)$, hence $\OO^r \rightarrow F$ gives only $r(F)$ sections, i.e., it factors through $\OO^{r(F)} \rightarrow F$. By part (\ref{one}), we have a short exact sequence
    \[0 \rightarrow \OO^{r(F)} \rightarrow F \rightarrow Q_F \rightarrow 0\] 
where $Q_F$ is supported on a curve $ C'$, and $\deg C' \leq c_1(F) < c_1 = \deg C$. Now we have a diagram
    \begin{center}
   \begin{tikzcd}
    0 \arrow[r] &\OO^r \arrow[r] \arrow[d] & E \arrow[r] \arrow[d] & \OO_C(D)  \arrow[r] \arrow[d] & 0\\ 
    0 \arrow[r] & \OO^{r(F)} \arrow[r] & F \arrow[r] & Q_F \arrow[r] & 0 
   \end{tikzcd}
   \end{center}
where the vertical maps are surjective. The map $\OO_C(D) \rightarrow Q_F$ induces an inclusion $\text{Supp } Q_F \subseteq C$. $\text{Supp } Q_F$ has strictly smaller degree than $C$. Since $C$ is irreducible, this is impossible. We conclude $Q_F = 0$, so $F$ is a slope 0 semistable sheaf with $r(F)$ sections. Hence $F = \OO^{r(F)}$, and the exact sequence defining $E$ is partially split. 

We count dimensions to show such extensions form a closed locus in the associated extension space. If $\deg D \leq \frac{1}{2}(\deg(C)^2+3\deg(C))-r$, then $\ext^1(\OO_C(D), \OO_{\PP^2}) \geq r$. If $E = \OO_{\PP^2}^q \oplus E'$ where $ E' \in \Ext^1(\OO_C(D), \OO_{\PP^2}^{r-q})$, then the dimension of the locus of such $E \in \Ext^1(\OO_C(D), \OO_{\PP^2}^r)$ is no larger than 
    \[\dim G(q,r) + \ext^1(\OO_C(D), \OO^{r-q}) = q(r-q) + (r-q)e,\] 
and $\ext^1(\OO_C(D), \OO_{\PP^2}^r) = re$. Now $q(r-q)+(r-q)e -re =(r-e-q)q$, so we see that when $e \geq r$ and $q > 0$, 
    \[\dim G(q,r)+ \ext^1(\OO_C(D), \OO_{\PP^2}^{r-q}) < \ext^1(\OO_C(D), \OO_{\PP^2}^r).\] 
The lemma follows.
\end{proof}

\begin{theorem} \label{depth}
Suppose that $\vv \in K(\PP^2)$ is stable and $\ch_0(\vv) \geq p(\ch_1(\vv))$, $\ch_1(\vv)>0$. Then $B^r(\vv)$ is nonempty and contains a component of the expected dimension.
\end{theorem}

\begin{proof}
By Lemma \ref{corank0lemma} (\ref{one}), we get a morphism $\psi: B^r(\vv) \rightarrow \PP H^0(\OO_{\PP^2}(\ch_1(\vv))) = \PP V$, by sending $E \in M(\vv)$ with $r$ global sections $s_1, ..., s_r \in H^0(E)$ to $s_1 \wedge \cdots \wedge s_r \in \PP V$, and by Lemma \ref{corank0lemma} (\ref{three}), $\psi$ is surjective. Denoting by $U = \PP V \smallsetminus \Gamma$ the locus of smooth curves, we see that $\psi$ factors through the relative Picard scheme $\Pic^d_{\mathcal{C}/U}$ over the universal curve $\mathcal{C} \rightarrow U$, where $d = \deg(D)$. The induced map $\phi: B^r(\vv) \rightarrow \Pic^d_{\mathcal{C}/U}$ is surjective, and since $E$ uniquely determines $Q_E = \OO_C(D)$ its fibers are quotients of open subsets of the extension spaces $\Ext^1(\OO_C(D), \OO_{\PP^2}^r)$. It follows in particular that $\psi^{-1}(U) \subseteq B^r(\vv)$ is irreducible and open, so its closure forms a nonempty irreducible component. 

We now compute its dimension. Write $r = \ch_0(\vv)$, $c_1 = \ch_1(\vv)$, $\mu = c_1/r$, and $\Delta(\vv) = \Delta_0 + k/r$ where $\Delta_0$ is the minimal discriminant of a stable sheaf with $\mu = \mu(\vv)$ and $r = \ch_0(\vv)$. Note in particular that $\vv$ lies directly above the $^{\perp}\OO$-branch of the Dr\'ezet-Le Potier curve, which is given by
    \[\Delta = p(-\mu) = \frac{1}{2}\mu^2 - \frac{3}{2}\mu + 1.\]
In particular, the characters $(\mu, \delta(\mu))$ are integral, so $\delta(\mu) = \Delta_0$ for these slopes. Then $M(\vv)$ is irreducible of dimension 
    \[\dim M(\vv) = r^2(2(\Delta_0 + k/r)-1)+1 = 1 + c_1^2 - 3c_1 r + r^2 + 2kr.\]
The expected codimension of the Brill-Noether locus $B^m(\vv)$ is $m(m-\chi(\vv))$. Now 
    \[\chi(\vv)= r \left(p\left(\mu\right) - \Delta_0 - k/r \right)= 3c_1 - k,\]
so the expected codimension of $B^r(\vv)$ is $r(r-3c_1+k)$, and its expected dimension is
    \[\expdim B^r(\vv) = (1 + c_1^2 -3c_1r + r^2 + 2kr) - r(r - 3c_1 + k) = 1 + c_1^2 + kr.\]
The dimension of the locus of interest is equal to $\dim \phi^{-1} \Pic^d(\mathcal{C}/U)$. The general $E \in \Ext^1(\OO_C(D), \OO^r)$ admits precisely one map $\OO^r \rightarrow E$, and automorphisms of $\OO^r$ determine distinct extension classes but isomorphic extension sheaves. Thus
    \[\dim \phi^{-1} \Pic^d_{\mathcal{C}/U} = \dim U + g(C) + \ext^1(\OO_C(D), \OO^r) - r^2 + 1.\]

We have $\ext^2(\OO_C(D), \OO)= h^0(\OO_C(D-3H))=0$,
and clearly $\hom(\OO_C(D), \OO) = 0$, so $\ext^1(\OO_C(D), \OO)= -\chi(\OO_C(D), \OO)$.
From the exact sequence 
    \[0 \rightarrow \OO^r \rightarrow E \rightarrow \OO_C(D) \rightarrow 0\]
we obtain 
    \[\chi(\OO_C(D), \OO )= \chi(E, \OO )-\chi(\OO^r, \OO)=-k-r.\]
Hence 
    \[\dim \phi^{-1}(\OO_C(D))= r(k+r)-r^2 = kr.\]
As $\dim U= \dim \PP H^0(\OO_{\PP^2}(c_1))= \frac{1}{2}(c_1^2+3c_1)$, we have
    \[\dim \Pic^d_{\mathcal{C}/U}= \dim U + g(C)=(c_1^2+3c_1)/2 + (c_1^2-3c_1+2)/2=c_1^2+1,\]
hence 
    \[\dim \psi^{-1}(U) = c_1^2 + 1 + kr,\]
which equals the expected dimension. 
\end{proof}

When we have instead $\ch_0(\vv) < p(\ch_1(\vv))$, we prove a weaker statement: that the Brill-Noether locus $B^r(\vv)$ is nonempty.

Note that the condition $\chi(\vv) = \ch_0(\vv)$ can be expressed via Riemann-Roch as the polynomial condition
    \[\Delta(\vv) = p(\mu(\vv)) - 1 = \frac{1}{2}\mu(\vv)^2 + \frac{3}{2}\mu(\vv).\]
This defines a parabola $\xi_r$ in the $(\mu, \Delta)$-plane. ($\xi_r$ does not depend on the rank.)

In the region of the $(\mu,\Delta)$-plane with $\mu > 0$, the parabola $\xi_r$ meets the Dr\'ezet-Le Potier curve only at the $^{\perp}\OO$ branch, at the point $(\mu, \Delta) = (1/3, 5/9)$. The Chern characters $\vv \in K(\PP^2)$ lying on $\xi_r$ satisfy $\chi(\vv) = \ch_0(\vv)$, and a general sheaf $E \in M(\vv)$ has $h^0(E) = r(E)$.

\begin{theorem} \label{nonempty}
Let $\vv \in K(\PP^2)$ be a stable Chern character with $\mu(\vv) > 0$. The Brill-Noether locus $B^r(\vv)$ is nonempty.
\end{theorem}

\begin{proof}
When $r=1$, $c_1\geq 1$, we may take $I_{Z}(c_1)$ where $Z$ lies on a curve of degree $c_1$. When $c_1=1$, by Theorem \ref{depth} the theorem is true for $r \geq 3$. Hence in the following we assume $r>1$, $c_1>1$ or $c_1=1, r=2$. We separate into two cases. 

\textit{Case I.} In the first case we assume $0 < \mu(\vv) \leq 1/3$; in particular, $\vv$ lies above the $^{\perp}\OO$-branch of the Dr\'ezet-Le Potier curve. We induct on the rank for a stronger statement: there exists $E \in  M(\vv)$ such that $h^0(E)=\ch_0(E)$. When the rank is 1 or 2 the statement is vacuous. When $r=3$, $c_1=1$, by Theorem \ref{depth} the theorem is true. 

Let $\vv'$ be the associated extremal character, and $\vv'' = \vv - \vv'$ the quotient character. Then $0 \leq \mu(\vv') \leq \mu(\vv)\leq 1/3$. We have $\mu(\vv')=0$ if and only if $c_1=1$, where the only possiblity for the theorem to be false is $r=2$, which is not considered in this case, hence $0< \mu(\vv')\leq 1/3$.

We also have $0 < \mu(\vv'') \leq 1/3$. To see this, first consider the case where $\mu(\vv') = \mu(\vv)$, which occurs if and only if $\ch_0(\vv)$ and $\ch_1(\vv)$ are not coprime. In this case, $\mu(\vv'') = \mu(\vv) \leq 1/3$. Otherwise, $\ch_0(\vv)$ and $\ch_1(\vv)$ are coprime. If $\mu(\vv) = 1/3$, then $\ch_0(\vv) = 3$ and $\ch_1(\vv) = 1$, which has already been addressed. In the remaining cases, $\ch_0(\vv) \geq 4$. In this case $\mu(\vv'')$ is the right Farey neighbor of $\mu(\vv)$, which is no greater than $1/3$.


An immediate consequence is that the extremal decomposition exists in this case. Denote the Dr\'ezet-Le Potier curve by $\Delta=\delta(\mu)$. If $\mu(\vv')< \mu(\vv)$, then $\Delta(\vv')\leq \delta(\mu(\vv'))$ since $(r(\vv'), \mu(\vv'), \delta(\mu(\vv')))$ is integral ($\vv'$ is either on the Dr\'ezet-Le Potier curve or semi-exceptional). Then $\chi(\vv', \OO)\geq 0$ (with equality only if $\vv' = \ch(\mathcal{O}^k)$), and $\chi(\vv, \OO)\leq 0$, hence $\chi(\vv'', \OO)=\chi(\vv, \OO)-\chi(\vv', \OO)\leq 0$, so $\vv''$ lies on or above the Dr\'ezet-Le Potier curve since $0< \mu(\vv'')\leq 1/3$. If $\mu(\vv')=\mu(\vv)$, we may write $\vv=(mr_{0}, \mu, \delta(\mu)+\frac{k}{mr}) $ where $\ch_1(\vv)=md_{0}$, $\mu=d/r_{0}$, and $r_{0}, d_{0}$ are coprime. Let $\vv'_{0}:=(r_{0}, \mu, \delta(\mu)+\frac{k}{r})$, $\vv''_{0}:=((m-1)r_{0}, \mu, \delta(\mu))$, one may check that $\vv'_{0}$ and $\vv''_{0}$ are indeed integral. Then $\vv=\vv'_{0}+\vv''_{0}$. Now clearly $\Delta(\vv'_{0})\geq \Delta(\vv)$, hence by definition of extremal character, $\Delta(\vv')\leq \Delta(\vv'_{0})$, hence $\Delta(\vv'')\geq \Delta(\vv''_{0})=\delta(\mu)$, so $\vv''$ also lies on or above the Dr\'ezet-Le Potier curve.

By assumption $r \geq 2$, both $\vv'$ and $\vv''$ have smaller rank than $\vv$, thus by induction we can assume that there are $E' \in M(\vv')$ and $E'' \in M(\vv'')$ satisfying $h^0(E') = r(E')$ and $h^0(E'') = r(E'')$, and by Proposition \ref{extremal} any nontrivial extension sheaf
    \[0 \rightarrow E' \rightarrow E \rightarrow E'' \rightarrow 0\]
is semistable. To show $h^0(E) = r(E) = r(E') + r(E'')$, it is enough to show that we can find such an extension such that the associated connecting homomorphism $H^0(E'') \rightarrow H^1(E')$ vanishes.

Indeed, the association $e \mapsto \delta_e$ determines a linear map 
    \[\Ext^1(E'', E') \rightarrow \Hom(H^0(E''),H^1(E')).\]
We need to show that it has nonzero kernel; in fact, the dimension of the source exceeds the dimension of the target. By stability and Serre duality, $\ext^2(E'', E') =\hom(E',E''(-3))= 0$, so
    \[\ext^1(E'', E') \geq -\chi(E'', E') = -r'r''(p(\mu' - \mu'') - \Delta' - \Delta''),\]
and we know already that $h^0(E'') = r''$, and
    \[h^1(E') = h^0(E') - \chi(E') = r' - r'(p(\mu') - \Delta').\]
Since $\vv''$ lies above the $^{\perp}\OO$-branch of the Dr\'ezet-Le Potier curve,
$\chi(E'', \OO) \leq 0$. It follows in particular from Riemann-Roch that $\Delta'' \geq p(-\mu'')$. We then have:
    \begin{align*}
        \ext^1(E'', E) - \hom(H^0(E''), H^1(E')) &\geq -r'r''(p(\mu' - \mu'') - \Delta' - \Delta'') \\& \qquad - r''(r' -r'(p(\mu')-\Delta')) \\
        &= -r'r''(p(\mu' - \mu'')-p(\mu') - \Delta''+1) \\
        &\geq -r'r''(p(\mu' - \mu'')-p(\mu') - p(-\mu'') + 1) \\
        &= r'r''\mu'\mu'' = c_1'c_1'' > 0
    \end{align*}
as desired. We conclude there is $E \in M(\vv)$ with 
    \[h^0(E) = h^0(E') + h^0(E'') = r(E') + r(E'') = r(E),\]
i.e., $B^r(\vv) \neq \emptyset$.

\textit{Case II.} In the second case, we assume $\mu(\vv) > 1/3$. Let $\vv \in K(\PP^2)$ be a stable Chern character of slope $\mu = \mu(\vv)$ and rank $r = \ch_0(\vv)$, and let $\Delta_0$ be the minimum discriminant of a stable character of slope $\mu$ and rank $r$. Then we may write $\Delta(\vv) = \Delta(\vv_0) + k/r$ for some $k \geq 0$. For every rational number $\mu' = c_1'/r'$, the character $\vv' = (r',\mu',\xi_r(\mu))$ is integral since $\chi(\vv') = r'$ is an integer. 

\begin{figure}[htb]
	\centering
	\includegraphics{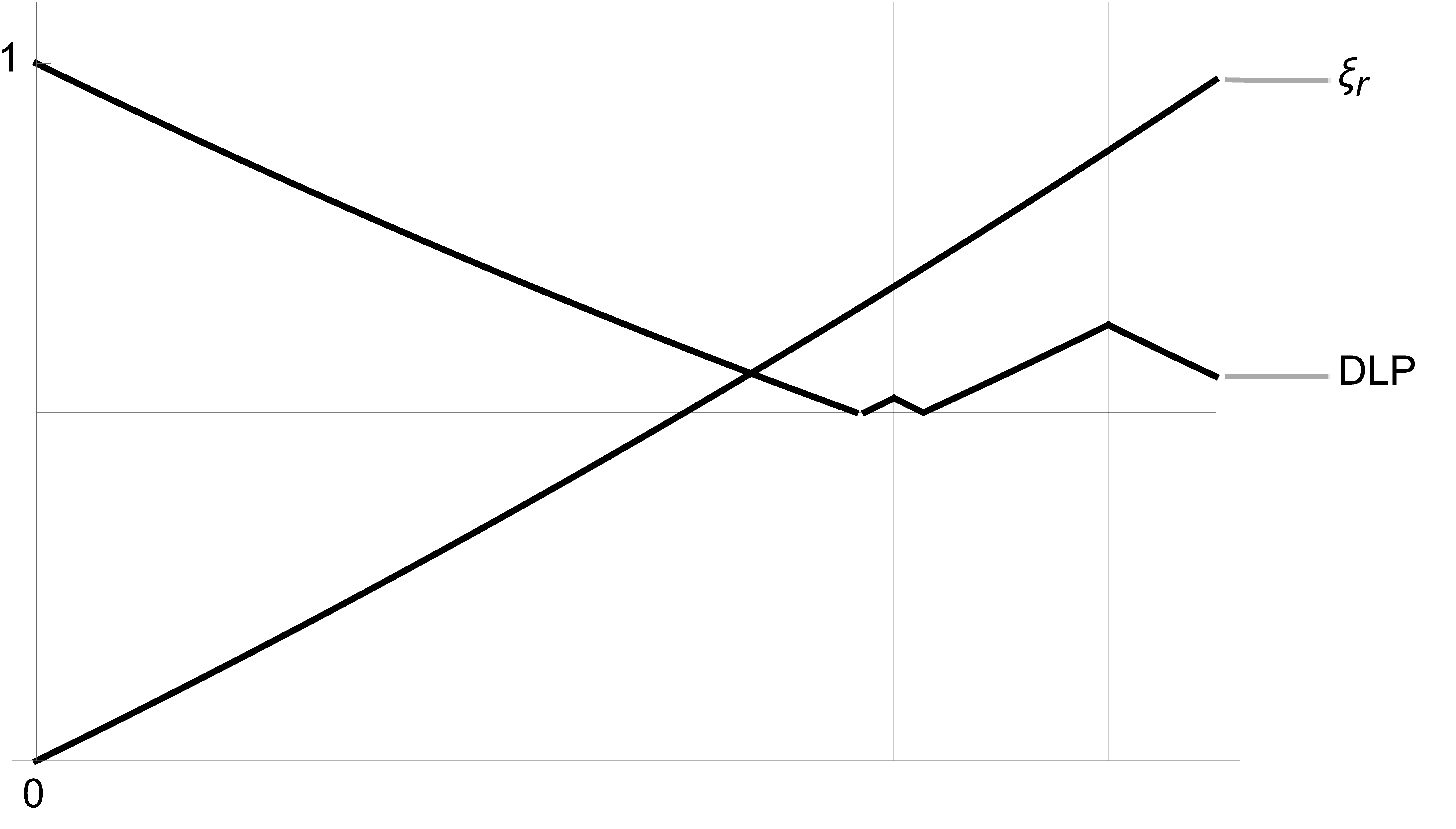}
	\caption{The regions verifying $\xi_r(\mu) - \delta(\mu) > 1/r$.}
	\label{fig:mindisc}
	\label{fig:coords}
\end{figure}

We claim that the minimal-discriminant character $\vv_0 = (\mu(\vv), \Delta_0)$ lies below the parabola $\xi_r$ in the $(\mu, \Delta)$-plane, so $\chi(\vv_0) > r$. To check this we will show that $\xi_r(\mu) - \delta(\mu) > 1/r$, so that any character with discriminant larger than $\xi_r(\mu)$ is obtained from one below $\xi_r$ by an elementary modification, so is not minimal. We divide the $(\mu, \Delta)$-plane into smaller regions, see Figure \ref{fig:mindisc}: first, it is enough to check this for $1/3 < \mu < 1/2$, since the Dr\'ezet-Le Potier curve is symmetric along the line $\mu = 1/2$. For slopes $\mu$ lying directly above the $^{\perp}\OO$-branch of the Dr\'ezet-Le Potier curve, one can easily check that the points $(\mu, \delta(\mu))$ are integral, so the claim is true here. This branch terminates at $\mu = (3 - \sqrt{5})/2$ (see \cite[\S 2.3]{CHW}); the branch $T_{\PP^2}(-1)^{\perp}$ to the left of the line $\mu = 1/2$ terminates at $\mu = 1/2 - (3-\sqrt{8})/2$. In this region we have the coarse bounds
    \begin{align*}
        \xi_r(\mu) - \delta(\mu) &\geq \xi_r(\mu) - (\text{peak of DLP curve above } E_{2/5}) \\
        &= \xi_r(\mu) - \frac{13}{25} \geq \xi_r\left( \frac{3-\sqrt{5}}{2} \right) - \frac{13}{25} \approx .12589
    \end{align*}
Thus if $\xi_r(\mu) - \delta(\mu) \leq 1/r$ then $r < 9$, and the few slopes $\mu = c_1/r$ with $r < 9$ in this region can be easily checked to have minimal discriminants lying below the parabola $\xi_r$.

Now for slopes $\mu$ lying directly above the $T_{\PP^2}(-1)^{\perp}$-branch, we have
    \[\xi_r(\mu) - \delta(\mu) = \xi_r(\mu) - \left( \frac{1}{2}\left( \mu-\frac{1}{2} \right)^2 + \frac{3}{2}\left( \mu - \frac{1}{2} \right) + 1 \right) = \frac{\mu}{2}\]
and $\mu/2 \leq 1/r$ only when $c_1 \leq 2$. All statements are clear when $c_1 = 1$, and when $r \geq 5$, $\mu = 2/r$ is to the left of this region, and for smaller ranks the minimal discriminant is easily computed to be below $\xi_r(\mu)$.

By generic slope-stability (\cite[Corollaire 4.12]{DLP}), a general member of $M(\vv_0)$ is slope-stable. Choose such a $E_0 \in M(\vv_0)$ with $r$ independent sections $s_1, ..., s_r \in H^0(E_0)$; choose a codimension two subscheme $Z$ of length $k$. Then the elementary modification $E$ defined by the short exact sequence
    \[0 \rightarrow E \rightarrow E_0 \rightarrow \OO_Z \rightarrow 0\]
has Chern character $\vv$. Since $E_0$ is slope-stable, so is $E$ (see Section \ref{jumpbyr}). Furthermore we can choose the map $E_{0} \rightarrow \OO_{Z}$ such that the map $H^0(E_0) \rightarrow H^0(\OO_Z)$ vanishes, since the dependency locus of $s_1, \cdots, s_{r}$ has dimension at least 1. So $h^0(E) \geq r$ for all $E$ constructed in this way.
\end{proof}

\section{Irreducibility \& reducibility of Brill-Noether loci} \label{redsection}
In this section we address the reducibility of the Brill-Noether loci $B^k(\vv)$. We show that when $\ch_1(\vv) = 1$, all Brill-Noether loci are irreducible and of the expected dimension. When $\ch_1(\vv) > 1$, we give examples where the Brill-Noether loci are reducible, and describe their components.

\begin{theorem} \label{irred}
Suppose $\vv \in K(\PP^2)$ is a stable Chern character with $\ch_1(\vv) = 1$. Then all of the Brill-Noether loci $B^k(\vv)$ are irreducible and of the expected dimension.

Each Brill-Noether locus $B^k(\vv)$ is the quotient of a projective bundle over some $M(\ww)$ by a free action of a projective linear group, where $\ch_1(\ww) = 1$ as well.
\end{theorem}

\begin{proof}
Choose $E \in M(\vv)$ and set $r = r(E)$. Consider the image sheaf $0 \neq F \subseteq E$ of the evaluation map on global sections for $E$. The proof of Theorem \ref{pbound}, Case I implies that if $r \geq 3$, then $h^0(E) \leq r$, and that the evaluation map in this case is a sheaf-theoretic embedding. (In rank 2, see Remark \ref{rmkdegone}. The only sheaf whose evaluation map on global sections is not an embedding is $T_{\PP^2}(-1)$, and here the statements are vacuous.) Thus every $E \in M(\vv)$ is given by an extension
    \begin{equation} \label{extension}
        0 \rightarrow \OO^{r-s} \rightarrow E \rightarrow E' \rightarrow 0
    \end{equation}
where $h^0(E) = r-s$, $s \geq 0$, and $E'$ has $c_1(E') = c_1(E) = 1$ and $r(E') = s$. In fact when $s \geq 0$ all such $E'$ are semistable, and when $s > 0$, torsion-free. To check torsion-freeness, consider the torsion part $T$ of $E'$, which sits in an exact sequence
    \[0 \rightarrow T \rightarrow E' \rightarrow E'/T \rightarrow 0.\]
Then arguing as in the proof of Lemma \ref{corank0lemma} (\ref{one}), we see that $T$ is supported in dimension 1. Then $c_1(E'/T) = c_1(E') - c_1(T) \leq 0$, and the quotient $E \rightarrow E'/T$ destabilizes $E$. Now for stability, if $s \geq 1$ then $c_1(E')$ and $r(E') = s$ are coprime, any destablizing quotient $Q$ of $E'$ would have $\mu(Q)<\mu(E')$, but since $c_{1}(E')=1$, we have $\mu(Q)\leq 0$, so $Q$ would be a destablizing quotient of $E$. When $s = 0$, by Lemma \ref{corank0lemma} (\ref{two}), $E'$ is of the form $\OO_L(a)$ for some $a \in \Z$ and $L \subseteq \PP^2$ a line, hence stable. Set $\ww = \ch(E') = \ch(E) - (r-s) \ch(\OO)$. 

We form extension parametrizations $\phi_{r-s}: \PP_{r-s} \dashrightarrow M(\vv)$ whose image lies in $B_{r-s - \chi(\vv)}(\vv)$ as follows. The expected extensions (\ref{extension}) are determined by classes in $\Ext^1(E', \OO^{r-s})$. Since $\Hom(E', \OO^{r-s})=\Ext^2(E',\OO^{r-s})=0$, $\ext^1(E', \OO^{r-s})=-\chi(E',\OO^{r-s})$ is constant, so the extension parametrization is defined over the whole moduli space $M(\ww)$. $\PP_{r-s}$ maps rationally to $B_{r-s - \chi(\vv)}(\vv) \subseteq M(\vv)$.

When $s > 0$, Proposition \ref{extremal} implies that for any $F' \in M(\ww)$ any nonzero extension $0 \neq [F] \in \Ext^1(F', \OO^{h^0(E)})$ is stable. When $s = 0$, Lemma \ref{corank0lemma} (\ref{three}) implies the same statement. In particular, each of the extension parametrizations $\phi_k: \PP_k \rightarrow M(\vv)$ is a morphism. Since each $E'$ as above is stable, it follows that $\phi_k$ surjects onto $B^k(\vv)$. The moduli spaces $M(\vv')$ are irreducible, so each $\PP_k$ is as well. It follows that so is each $B^k(\vv)$. We now show that they are of the expected dimension.

We will need to know the other numerical invariants of $E$ and $E'$. They are determined as follows. Set $\ch_2(E) = a - 1/2 = \ch_2(E')$ (see Remark \ref{a}). This gives
    \[\Delta(E) = \frac1{2r^2} - \frac{a}r + \frac1{2r}, \quad \chi(E) = r(p(1/r) - \Delta(E)) = r + a + 1.\]
Similarly,
    \[\Delta(E') = \frac1{2s^2} - \frac{a}{s} + \frac1{2s}.\]
The dimension of $\PP$ is
    \[\dim \phi(\PP) = \dim M(\ww) + \ext^1(E', \OO^{r-s})-1.\]
The latter quantity is, by Serre duality, Theorem \ref{GH}, and Riemann-Roch,
    \begin{align*}
        \ext^1(E', \OO^{r-s}) &= (r-s) \cdot h^1(E'(-3)) \\
        &= -(r-s) \cdot \chi(E'(-3)) \\
        &= -(r-s) \cdot s(P(1/s-3) - \Delta(E')) \\
        &= -(r-s) \cdot (s+a-2). 
    \end{align*}

The dimensions of the moduli spaces are
    \[\dim M(\vv) = r^2(2\Delta(E) - 1) + 1 = -r^2 + (1-2a)r + 2\]
and 
    \[\dim M(\ww) = s^2(2\Delta(E') - 1) + 1 = -s^2 + (1-2a)s + 2.\]
The fibers of $\phi$ are of dimension $\Aut(\OO^{r-s}) = (r-s)^2-1$, so the dimension of the image of $\phi$ is
    \begin{align*}
        \dim \phi(\PP) &= s^2(2\Delta(E') - 1) + 1 - (r-s) \cdot (s+a-2) - (r-s)^2+1 \\
        &= -s^2 + (1-2a)s + 2 -(r-s) \cdot (s+a-2) - 1 - (r-s)^2 + 1\\
        &= -r^2 +(s-a+2)r - s^2 - (a+1)s + 2.
    \end{align*}

This is the actual dimension of the Brill-Noether locus; we now check that the expected dimension of the Brill-Noether loci equals the actual dimension.

Set $m = r - s - \chi(\vv)$; the expected codimension of the Brill-Noether locus $B_m(\vv)$ is 
    \begin{align*}
        \expcodim_m(\vv) &= (r - s) \cdot m \\
        &= -(r - s) \cdot (r - s - \chi(\vv)) \\
        &= (s-r) \cdot (s - a - 1) \\
        &= s^2 +(1-r+a)s -ar -r.
    \end{align*}
The expected dimension is then
    \begin{align*}
        \expdim_m(\vv) &= \dim M(\vv) - (s^2 +(1-r+a)s -ar -r) \\
        &= -r^2 + (1-2a)r + 2 - s^2 - (1-r+a)s + ar + r \\
        &= -r^2 + (s-a+2)r - s^2 - (a+1)s + 2
    \end{align*}
as required.

For the latter statement, we consider the map $\PP_{k} \rightarrow B^{k}(\vv)$. By Lemma 6.3 in \cite{CH2}, all fibers of this map are $\PP GL_{k}$.
\end{proof}

\begin{remark} \label{a}
When $h^0(E) = r$, the associated extension is of the form
    \[0 \rightarrow \OO^r \rightarrow E \rightarrow \OO_L(a) \rightarrow 0\]
for $L \subseteq \PP^2$ a line and $a \leq 0$. In this case $\ch_2(E) = a - 1/2$. We regard $a$ as an invariant of $E$ for other values of $h^0(E)$ via this formula.
\end{remark}

\subsection{Reducible Brill-Noether loci} \label{red}
When $\ch_1(\vv) > 1$, the Brill-Noether loci can be reducible.

When $\ch_0(\vv) = 1$, we consider twists of ideal sheaves $I_Z(d)$, where $d > 0$ and $Z \subseteq \PP^2$ is a subscheme of finite length. These moduli spaces are isomorphic to the Hilbert scheme $\PP^{2[n]}$, but their cohomological properties depend on the twist. We often write $\PP^{2[n]}(d)$ for these moduli spaces. 

\begin{proposition} \label{twist}
Let $\mathbf{v} = (1, d, d^2/2 - n)$ be the Chern character of a twist $I_Z(d)$ of an ideal sheaf of a scheme of length $n$. If $n > d^2$, the Brill-Noether loci $B^2(\mathbf{v})$ are reducible, and have at least $d-1$ components. 
\end{proposition}

\begin{proof}
Choose $Z$ such that $h^0(I_Z(d)) \geq 2$, and let $C_1, C_2 \subseteq \PP^2$ be curves of degree $d$ containing $Z$. If $C_1 \cap C_2$ were zero-dimensional it would have length $d^2 < \ell(Z)$, so in particular it could not contain $Z$. We conclude that $C_1 \cap C_2$ has one-dimensional components. The general such $Z$ lies in $B^2 \smallsetminus B^3$ so the pencil spanned by $C_1$ and $C_2$ is determined by $Z$. In this case we call $C_Z$ the union of the one-dimensional components of $C_1 \cap C_2$; the general such $Z$ will have $C_Z$ irreducible. 

For any $0 < e < d$, we examine the following loci in $B^2$:
    \[B^2_e := \overline{\{Z : C_Z \text{ is irreducible with } \deg(C_Z) = e\}}.\]
Clearly $B^2$ is the union of the $B^2_e$ as $e$ ranges between 1 and $d-1$; we claim that each lies in a distinct irreducible component of $B^2$.

Indeed, the association of a general element $I_Z(d) \in B^2$ to the pencil of reducible degree $d$ curves containing it induces a rational map $\phi: B^2 \dashrightarrow G(2, p(d))$, and the pencils themselves are lines on a Segre variety $\Sigma_{d-k,k} \subseteq \PP H^0(\OO_{\PP^2}(d))$ lying in a fiber of one of its projections $\Sigma_{d-k,k} \rightarrow \PP H^0(\OO_{\PP^2}(d-k)), \PP H^0(\OO_{\PP^2}(k))$. One can check that each such line (given by fixing a smooth curve of degree $d-k$ and a pencil of curves of degree $k$) lies in the image of $\phi$. Denote by $P_{d-k,k}$ the loci of such lines in $G(2, p(d))$.

As $k$ varies between 1 and $d-1$, there are no containments among the Segre varieties $\Sigma_{d-k,k}$ in $\PP H^0(\OO_{\PP^2}(d))$, and the line classes in the fibers of the projections form irreducible components of the image of $\phi$. Each of the loci $B^2_e$ dominates $P_{e, d-e}$, so the claim follows.
\end{proof}

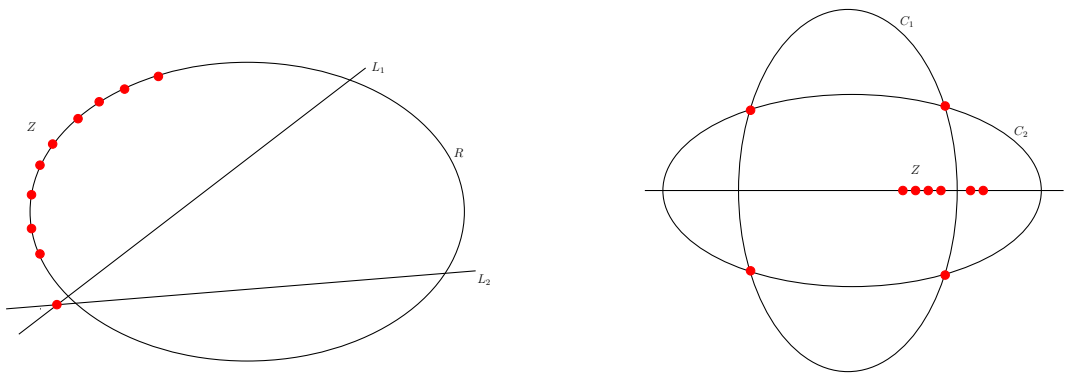
\begin{figure} 
    \centering
    \scalebox{.4}{
    \begin{tikzpicture}[x=0.4pt, y=0.4pt]
    \draw[solid, draw={rgb,255:red,0;green,0;blue,0}, draw opacity=1, line width=1, fill={rgb,255:red,0;green,0;blue,0}, fill opacity=0] (590,390) ellipse (513.4119409414459 and 353.41194094144583);
    \draw[solid, draw={rgb,255:red,0;green,0;blue,0}, draw opacity=1, line width=1, ] (50,100) -- (870,730);
    \draw[solid, draw={rgb,255:red,0;green,0;blue,0}, draw opacity=1, line width=1, ] (100,160) -- (100,160);
    \draw[solid, draw={rgb,255:red,0;green,0;blue,0}, draw opacity=1, line width=1, ] (20,160) -- (1130,250);
    \draw[solid, draw={rgb,255:red,255;green,0;blue,0}, draw opacity=1, line width=1, fill={rgb,255:red,255;green,0;blue,0}, fill opacity=1] (130,550) ellipse (10 and 10);
    \draw[solid, draw={rgb,255:red,255;green,0;blue,0}, draw opacity=1, line width=1, fill={rgb,255:red,255;green,0;blue,0}, fill opacity=1] (270,560) ellipse (0 and 0);
    \draw[solid, draw={rgb,255:red,255;green,0;blue,0}, draw opacity=1, line width=1, fill={rgb,255:red,255;green,0;blue,0}, fill opacity=1] (80,430) ellipse (10 and 10);
    \draw[solid, draw={rgb,255:red,255;green,0;blue,0}, draw opacity=1, line width=1, fill={rgb,255:red,255;green,0;blue,0}, fill opacity=1] (140,170) ellipse (10 and 10);
    \draw[solid, draw={rgb,255:red,255;green,0;blue,0}, draw opacity=1, line width=1, fill={rgb,255:red,255;green,0;blue,0}, fill opacity=1] (100,500) ellipse (10 and 10);
    \draw[solid, draw={rgb,255:red,255;green,0;blue,0}, draw opacity=1, line width=1, fill={rgb,255:red,255;green,0;blue,0}, fill opacity=1] (100,290) ellipse (10 and 10);
    \draw[solid, draw={rgb,255:red,255;green,0;blue,0}, draw opacity=1, line width=1, fill={rgb,255:red,255;green,0;blue,0}, fill opacity=1] (300,680) ellipse (10 and 10);
    \draw[solid, draw={rgb,255:red,255;green,0;blue,0}, draw opacity=1, line width=1, fill={rgb,255:red,255;green,0;blue,0}, fill opacity=1] (240,650) ellipse (10 and 10);
    \draw[solid, draw={rgb,255:red,255;green,0;blue,0}, draw opacity=1, line width=1, fill={rgb,255:red,255;green,0;blue,0}, fill opacity=1] (80,350) ellipse (10 and 10);
    \draw[solid, draw={rgb,255:red,255;green,0;blue,0}, draw opacity=1, line width=1, fill={rgb,255:red,255;green,0;blue,0}, fill opacity=1] (190,610) ellipse (10 and 10);
    \draw[solid, draw={rgb,255:red,255;green,0;blue,0}, draw opacity=1, line width=1, fill={rgb,255:red,255;green,0;blue,0}, fill opacity=1] (380,710) ellipse (10 and 10);
    \node at (900,730) [opacity=1] {\textcolor[RGB]{0,0,0}{$L_1$}};
    \node at (1150,230) [opacity=1] {\textcolor[RGB]{0,0,0}{$L_2$}};
    \node at (1090,530) [opacity=1] {\textcolor[RGB]{0,0,0}{$R$}};
    \node at (80,590) [opacity=1] {\textcolor[RGB]{0,0,0}{$Z$}};
    \draw[solid, draw={rgb,255:red,0;green,0;blue,0}, draw opacity=1, line width=1, fill={rgb,255:red,255;green,0;blue,0}, fill opacity=0] (2010,440) ellipse (258.4971769034255 and 428.4971769034255);
    \draw[solid, draw={rgb,255:red,0;green,0;blue,0}, draw opacity=1, line width=1, fill={rgb,255:red,0;green,0;blue,0}, fill opacity=0] (2020,440) ellipse (447.34255586866 and 227.34255586865999);
    \draw[solid, draw={rgb,255:red,0;green,0;blue,0}, draw opacity=1, line width=1, ] (1530,440) -- (2520,440);
    \draw[solid, draw={rgb,255:red,255;green,0;blue,0}, draw opacity=1, line width=1, fill={rgb,255:red,255;green,0;blue,0}, fill opacity=1] (2330,440) ellipse (10 and 10);
    \draw[solid, draw={rgb,255:red,255;green,0;blue,0}, draw opacity=1, line width=1, fill={rgb,255:red,255;green,0;blue,0}, fill opacity=1] (2240,240) ellipse (10 and 10);
    \draw[solid, draw={rgb,255:red,255;green,0;blue,0}, draw opacity=1, line width=1, fill={rgb,255:red,255;green,0;blue,0}, fill opacity=1] (2230,440) ellipse (10 and 10);
    \draw[solid, draw={rgb,255:red,255;green,0;blue,0}, draw opacity=1, line width=1, fill={rgb,255:red,255;green,0;blue,0}, fill opacity=1] (2300,440) ellipse (10 and 10);
    \draw[solid, draw={rgb,255:red,255;green,0;blue,0}, draw opacity=1, line width=1, fill={rgb,255:red,255;green,0;blue,0}, fill opacity=1] (1780,630) ellipse (10 and 10);
    \draw[solid, draw={rgb,255:red,255;green,0;blue,0}, draw opacity=1, line width=1, fill={rgb,255:red,255;green,0;blue,0}, fill opacity=1] (2140,440) ellipse (10 and 10);
    \draw[solid, draw={rgb,255:red,255;green,0;blue,0}, draw opacity=1, line width=1, fill={rgb,255:red,255;green,0;blue,0}, fill opacity=1] (2200,440) ellipse (10 and 10);
    \draw[solid, draw={rgb,255:red,255;green,0;blue,0}, draw opacity=1, line width=1, fill={rgb,255:red,255;green,0;blue,0}, fill opacity=1] (2240,640) ellipse (10 and 10);
    \draw[solid, draw={rgb,255:red,255;green,0;blue,0}, draw opacity=1, line width=1, fill={rgb,255:red,255;green,0;blue,0}, fill opacity=1] (1780,250) ellipse (10 and 10);
    \draw[solid, draw={rgb,255:red,255;green,0;blue,0}, draw opacity=1, line width=1, fill={rgb,255:red,255;green,0;blue,0}, fill opacity=1] (2170,440) ellipse (10 and 10);
    \node at (2150,840) [opacity=1] {\textcolor[RGB]{0,0,0}{$C_1$}};
    \node at (2420,580) [opacity=1] {\textcolor[RGB]{0,0,0}{$C_2$}};
    \node at (2540,440) [opacity=1] {\textcolor[RGB]{0,0,0}{$L$}};
    \node at (2170,490) [opacity=1] {\textcolor[RGB]{0,0,0}{$Z$}};
    \end{tikzpicture}

    }

    \caption{General members of the loci $B_2^1$ and $B_2^2$ on $\PP^{2[10]}(3)$.}
    \label{fig:tenpoints}
\end{figure}

We will use reducibility in rank 1 to form reducible Brill-Noether loci in higher rank via the Serre construction.

\begin{theorem}[\cite{HL}, 5.1.1]
Let $X$ be a smooth surface, $L, M \in \Pic(X)$ be line bundles on $X$, and $Z \subseteq X$ a local complete intersection subscheme of codimension 2. Then there is a locally free extension
    \[0 \rightarrow L \rightarrow E \rightarrow M \otimes I_Z \rightarrow 0\]
(and hence the general such extension is locally free), if and only if $Z$ satisfies the \textit{Cayley-Bacharach property} with respect to $L^{\vee}MK_X$: for any $Z' \subseteq Z$ of length $\ell(Z) - 1$, the map
    \[H^0(L^{\vee}MK_X \otimes I_Z) \rightarrow H^0(L^{\vee}MK_X \otimes I_{Z'})\]
is surjective. 
\end{theorem}

Assume that $Z \subseteq \PP^2$ satisfies the Cayley-Bacharach property for $\OO_{\PP^2}(b-a-3)$. Then the general extension
    \[0 \rightarrow \OO(a) \rightarrow E \rightarrow I_Z(b) \rightarrow 0\]
is locally free. To check that the general such $E$ is slope stable, it is enough to check that there are no maps $\OO_{\PP^2}(d) \rightarrow E$ with $d > \mu(E) = (a+b)/2$. Indeed, any destabilizing subbundle $S \subseteq E$ has rank 1, and since $E$ is locally free, this inclusion induces $S^{\vee \vee} \subseteq E^{\vee \vee} \simeq E$, and $S^{\vee \vee}$ is locally free of rank 1, i.e., a line bundle. 

\begin{example} \label{rank2}
 The extensions we consider are in the above form, with $(a,b) = (0, 3)$ and $\ell(Z) = 10$:
    \[0 \rightarrow \OO \rightarrow E \rightarrow I_Z(3) \rightarrow 0.\]
Here $\chi(E) = h^0(E) = 1$ for generic choice of $Z$. Set $\vv = \ch(E)$. Note that the Cayley-Bacharach condition is trivially satisfied, as the map in question is
    \[H^0(I_Z) = 0 \rightarrow 0 = H^0(I_{Z'}).\]
Stability is also easy to check. We need to rule out the existence of maps $\OO(d) \rightarrow E$ where $d \geq 2$. Any such map induces a nonzero map to from $\OO(d)$ to either $\OO$ or $I_Z(3)$, but either such map vanishes by stability when we choose $Z$ noncollinear. Thus the general extension $E$ is both locally free and slope-stable. We construct multiple components in $B^3(\vv)$ by considering the components of $B^2(I_Z(d))$. 

The Brill-Noether locus $B_2$ on the twisted Hilbert scheme $\PP^{2[10]}(3)$ has (at least) two components: $B^2_1$ and $B^2_2$. The general member $Z \in B^2_1$ is contained in two cubics $X$ and $Y$ each of which is the union of a fixed line $L$ with a conic $C_1$ or $C_2$. The general $Z \in B^2_2$ is contained in two cubics $U$ and $V$ each of which is the union of a fixed conic $R$ with a line $L_1$ or $L_2$. See Figure \ref{fig:tenpoints}.

We consider the following families of stable sheaves $E$: construct a projective bundle $P_1$ over $B^2_1(I_Z(d))$ with fiber $\PP \Ext^1(I_Z(3), \OO)$ for $Z \in B^2_1$. It is generically a family of stable sheaves, so admits a map $P_1 \dashrightarrow M(\ch(E))$; construct $P_2 \dashrightarrow M(\vv)$ similarly. One can check that the image closures $B_1$ and $B_2$ of $P_1$ and $P_2$ have dimensions 23, above the expected dimension 22 for a component of $B^2(\vv)$. 

To conclude that $B^3(\vv)$ is reducible, we need to verify that $B_1$ and $B_2$ are not nested, and there is no component of $B_3(\vv)$ containing both. The first statement is straightforward, following from the definitions of the loci. For the second, suppose $V \subseteq B^3(\vv)$ is irreducible and contains both $B_1$ and $B_2$. Then the general $E \in V$ has three global sections, and since the general points of $B_1$ and $B_2$ have global sections $\OO \rightarrow F$ with torsion-free cokernel, so too does the general point of $V$. Thus $E$ appears in an extension
    \[0 \rightarrow \OO \rightarrow E \rightarrow I_Z(3) \rightarrow 0\]
with $\ell(Z) = 10$. Since $E \in B^3(\vv)$, $h^0(I_Z(3)) \geq 2$ so $Z$ lies in $B^2_1$ or $B^2_2$. However a general point of $V$ specializing to a general point of $B_1$ or $B_2$ determines a specialization of $B^2_1$ to $B_2^2$, or vice versa. To see this, let $\mathcal{V}/V$ be a universal family, perhaps after shrinking $V$. Then we claim there is a distinguished section $\OO_{\PP^2 \times V} \rightarrow \mathcal{V}$ whose cokernel $\mathcal{I}$ is generically a family of ideal sheaves $I_Z(3)$, fitting in an exact sequence
    \[0 \rightarrow \OO_{\PP^2 \times V} \rightarrow \mathcal{V} \rightarrow \mathcal{I} \rightarrow 0\]
which on fibers recovers the sequences defining the loci $B_1$ and $B_2$. In particular, it will follow that a specialization of a general point of $V$ to a general point of $B_1$ and to a general point of $B_2$ determines a specialization of $B^2_1$ to $B^2_2$, contradicting Proposition \ref{twist}. To see this, let $E \in B_1$ be given, and choose a section $t \in H^0(E)$ fitting in an exact sequence
    \[0 \rightarrow \OO \stackrel{t}{\rightarrow} E \rightarrow I_{Z_t}(3) \rightarrow 0\]
for $Z_t \subseteq \PP^2$ a zero-dimensional subscheme of length 10 with $I_{Z_t}(3) \in B^2_1 \subseteq \PP^{2[10]}(3)$. Then choose general $s_1, s_2 \in H^0(E)$ independent from $t$. The section $s_1$ determines another sequence
    \[0 \rightarrow \OO \stackrel{s_1}{\rightarrow} E \rightarrow I_{Z_{s_1}}(3) \rightarrow 0\]
where $Z_{s_1} \subseteq \PP^2$ is a zero-dimensional subscheme, necessarily of length 10. We claim that $Z_{s_1} \in B^2_1$ as well. To see this, consider the dependency loci $X(s_1, t)$ and $X(s_1,s_2)$ where, respectively, $s_1$ and $t$ are dependent and where $s_1$ and $s_2$ are dependent. Each is a degree 3 curve in $\PP^2$, and $Z_{s_1} \subseteq X(s_1, t) \cap X(s_1, s_2)$. However since $\ell(Z_{s_1}) = 10$, the intersection cannot be transverse, and the dependency loci share a component. Since $Z_t \in B^2_1$, there is a line in the shared loci containing $Z_{s_1}$. The claim is now proved, which completes the example.
\end{example}

In fact in higher rank many Brill-Noether loci are reducible, as we now show. The context for the following result is established in Section \ref{boundssection}. Recall that the parabola $\xi_r$ is the locus of characters $\vv \in K(\PP^2)$ satisfying $\chi(\vv) = \ch_0(\vv)$, thought of in the $(\mu, \Delta)$-plane, and we set $\xi_r(a)$ to be the intersection of $\xi_r$ with the vertical line $\mu = a$.

\begin{remark}
In \cite[Definition 3.1]{CH2} a somewhat different definition of the extremal character to $\vv$ is given. We will need to use this other definition alongside ours in what follows; we set $\vv'_{CH}$ to be the extremal character in the sense of Coskun-Huizenga in \cite{CH2}. The character $\vv'_{CH}$ is defined as in Definition \ref{extremaldef} but replacing axioms (\hyperref[D1]{D1}) and (\hyperref[D2]{D2}) as follows:
\begin{enumerate}
    \item[(D1')] $\ch_0(\vv'_{CH}) \leq \ch_0(\vv)$ and if $\ch_0(\vv'_{CH}) = \ch_0(\vv)$ then $\ch_1(\vv) - \ch_1(\vv'_{CH}) > 0$;
    \item[(D2')] $\mu(\vv'_{CH}) < \mu(\vv)$.
\end{enumerate}

It will follow from Lemma \ref{nonint} that $\mu(\vv'_{CH}) \leq \mu(\vv')$. The variant $\vv_{CH}'$ has the property that when $\ch_0(\vv)$ and $\ch_1(\vv)$ are not coprime, we still have $\mu(\vv_{CH}') < \mu(\vv)$.
\end{remark}

\begin{lemma} \label{nonint}
	If $\gcd(r,c_1)>1$ and $\frac{c_1}{r} \not\in \Z$, then the extremal character $\vv_{CH}'$ has $r(\vv_{CH}')<r(\vv)$.
\end{lemma}

\begin{proof}
	If not, then $\mu(\vv')=\frac{c_1-1}{r}$. We claim that there exists an integer $n$ such that $\frac{c_1-1}{r} < \frac{n}{r-1} < \frac{c_1}{r}$. This is equivalent to the inequalities $\frac{(c_1-1)(r-1)}{r}< n < \frac{(c_1)(r-1)}{r}$. The sequence of $r$ consecutive integers 
	\[(c_1-1)(r-1), \quad (c_1-1)(r-1)+1, \quad ..., \quad c_1(r-1)\]
	contains exactly one entry divisible by $r$. By assumption, $\frac{c_1}{r} \notin \mathbb{Z}$, so $r$ does not divide $c_1(r-1)$. It suffices to prove that $r$ does not divide $(c_1-1)(r-1)$. If $r|(c_1-1)(r-1)$, then $r|(c_1-1)$. Set $m= \gcd(c_1,r)$. Then $m$ divides $r,c_1-1$, and $c_1$, which contradicts the assumption that $m>1$.
\end{proof}

\begin{theorem} \label{comps}
Let $\vv \in K(\PP^2)$ be a stable Chern character with $\mu(\vv) > 0$. Let $\mu_{CH}'$ denote the slope of the extremal character $\vv_{CH}'$ for $\vv$. If $\mu_{CH}' > 1/3$, $\Delta(\vv) \gg 0$ and $\mu(\vv) \notin \Z$, then $B^r(\vv)$ is reducible.
\begin{enumerate}
    \item If $\Delta(\vv) \geq \xi_r(\mu(\vv))$ and $\mu_{CH}' \geq 1/3$, then $B^r(\vv)$ contains a component of the expected dimension.
    \item If $\Delta(\vv) \gg 0$ and $\mu_{CH}' > 1/3$, then $B^r(\vv)$ contains a component of dimension larger than the expected dimension.
\end{enumerate}
\end{theorem}

Note that the assumptions in the theorem are implied by the assumptions in Theorem \ref{main} (4).

We first construct the families of sheaves necessary for the proof and study their properties. Our first family will consist of sheaves $E$ appearing in extensions determined by $r(E)$ global sections of $E$, and our second family will consist of sheaves $F$ appearing in extension classes determined by the extremal subsheaves $E'$. 

First, let $\Pic_{\mathcal{C}/U}^d$ be the relative Picard scheme over the universal curve $\mathcal{C} \rightarrow U$, where $U \subseteq \PP H^0\OO_{\PP^2}(c_1)$ is the space of smooth plane curves of degree $c_1$. We set $g = \frac{1}{2}(c_1-1)(c_1-2)$. Consider the projective bundle $Y$ whose fiber over $\OO_C(D)$ is $\PP\Ext^1(\OO_C(D), \OO^r)$ defined over the open subset of $\Pic_{\mathcal{C}/U}^d$ where the dimension of this group is constant. We form a family of sheaves $\mathcal{E}$ on $\PP^2$ parametrized by $Y$ by setting $E_y$, $y \in Y$, to be an extension sheaf
    \[0 \rightarrow \OO^r \rightarrow E_y \rightarrow \OO_C(D) \rightarrow 0\]
determined by $y = (\OO_C(D), e)$ for $\OO_C(D) \in \Pic^d_{\mathcal{C}/U}$ and $e \in \PP \Ext^1(\OO_C(D), \OO^r)$.

Recall that a sheaf $E$ on $\PP^2$ is called \textit{prioritary} if $\Ext^2(E, E(-1)) = 0$. 

\begin{lemma} \label{prioritary}
Let $\mathcal{E}$ be the family of sheaves just constructed. Then for $d = g - 1$,
\begin{enumerate}
    \item \label{lone} $\mathcal{E}$ is complete;
    \item \label{ltwo} the general member of $\mathcal{E}$ is prioritary. 
\end{enumerate}
When $c_{1}/r\geq 1/3$, the general member is stable.
\end{lemma}

\begin{proof}
(\ref{lone}) First, we have a natural identification
    \[T_{\OO_C(D)} \Pic^d_{\mathcal{C}/U} \simeq \Ext^1(\OO_C(D), \OO_C(D)).\] 
Let $T_y$ be the tangent space to $Y$ at $y \in Y$, and let $E = E_y$ be the associated sheaf. The Kodaira-Spencer map $\kappa: T_y \rightarrow \Ext^1(E,E)$ fits into the following diagram:
    \begin{center}
        \begin{tikzcd}
            0 \arrow[r] & \Ext^1(\OO_C(D), \OO^r) \arrow[r] \arrow[d] & T_y \arrow[r] \arrow[d, swap, "\kappa"] & \Ext^1(\OO_C(D), \OO_C(D)) \arrow[r] \arrow[d] & 0 \\
             & \Ext^1(E, \OO^r) \arrow[r] & \Ext^1(E,E) \arrow[r] & \Ext^1(E, \OO_C(D)) & 
        \end{tikzcd}
    \end{center}
The left-hand vertical map is obtained by applying $\Hom(-, \OO^r)$ to the sequence
    \[0 \rightarrow \OO^r \rightarrow E \rightarrow \OO_C(D) \rightarrow 0.\]
Since $\Ext^1(\OO^r, \OO^r) = 0$, it is surjective. The right-hand vertical map is obtained by applying $\Hom(-, \OO_C(D))$ to the same sequence. The next term in the sequence is
    \[\Ext^1(\OO^r, \OO_C(D)) \simeq H^1(C, \OO_C(D))^r\]
whose dimension is
    \[\ext^1(\OO^r, \OO_C(D)) = r \cdot h^1(C, \OO_C(D)) = r \cdot h^0(C, \OO_C(K_C - D).\]
We have
    \[\deg(K_C - D) = 2g-2 - d = g-1.\]
For general $D$, Riemann-Roch on $C$ then gives $h^0(\OO_C(K_C-D)) = 0$, so the right-hand vertical map in the diagram is also surjective. The snake lemma implies that $\kappa$ is surjective as well, i.e., $\mathcal{E}$ is complete.

(\ref{ltwo}) Apply $\Hom(E, -)$ to the sequence
    \[0 \rightarrow \OO(-1)^r \rightarrow E(-1) \rightarrow \OO_C(D-H) \rightarrow 0\]
to get a sequence
    \begin{equation} \label{prior}
        \Ext^2(E, \OO(-1)^r) \rightarrow \Ext^2(E, E(-1)) \rightarrow \Ext^2(E, \OO_C(D-H)) \rightarrow 0.
    \end{equation}
To show general $E$ is prioritary, it is enough to show that the outer two terms generically vanish. 

The first term is Serre dual to $\Hom(\OO(-1), E(-3))^r \simeq H^0(E(-2))^r$. This space sits in an exact sequence
    \[0 \rightarrow H^0(\OO(-2)^r) \rightarrow H^0(E(-2)) \rightarrow H^0(\OO_C(D-2H)) \rightarrow 0.\]
The first term vanishes, and as above for general $D$, $h^0(\OO_C(D-2H)) \leq h^0(\OO_C(D)) = 0$. We conclude the first term in the sequence (\ref{prior}) vanishes.

The last term in sequence (\ref{prior}) is Serre dual to $\Hom(\OO_C(D-H), E(-3))$. To show this vanishes, it is enough to show that generic $E$ appearing in $\mathcal{E}$ is torsion-free. If $T \subseteq E$ is the torsion part of $E$, write $E' = E/T$. The induced map $T \rightarrow \OO_C(D)$ is nonzero, so we obtain a diagram
    \begin{center}
        \begin{tikzcd}
        & 0 \arrow{d} & 0 \arrow[d] & \\
        K \arrow{r} & T \arrow{r} \arrow{d} & \mathcal{O}_{C}(D) \arrow{r} \arrow[d, equals] & Q \\
        \mathcal{O}^{r}  \arrow{r}  & E \arrow{r} \arrow{d} & \mathcal{O}_{C}(D) \arrow{r}\arrow{d} & 0 \\
        E' \arrow[r, equals]  & E'\arrow{r} \arrow{d} & 0 \arrow{r}\arrow{d} & 0 \\
        & 0 & 0 &        
        \end{tikzcd}
    \end{center}
The snake lemma produces an exact sequence
    \[0 \rightarrow K \rightarrow \OO^r \rightarrow E' \rightarrow Q \rightarrow 0.\]
Since $\OO^r$ is torsion-free, $K = 0$. Arguing in similar fashion to the proof of Lemma \ref{corank0lemma} (\ref{one}), we see that the inclusion $T \hookrightarrow \OO_C(D)$ cannot have a nonzero cokernel, which would be supported in codimension two. Thus $T \simeq \OO_C(D)$, but then the map $T \rightarrow E$ splits the sequence defining $E$. This contradiction implies that $E$ is torsion-free, as required. We conclude that the general member of the family $\mathcal{E}$ is prioritary.

Now $\vv(E)$ lies on the curve $\xi_{r}$. When $\mu(E)\geq 1/3$, $\vv(E)$ lies on or above the Dr\'ezet-Le Potier curve. To conclude that the general member of $\mathcal{E}$ is also stable, one proceeds in entirely the same way as in \cite[\S 16.2]{LP}. Specifically, the stack of prioritary sheaves is irreducible and its general member is stable. The family $\mathcal{E}$ induces a map to the stack of prioritary sheaves, and completeness implies the general member of $\mathcal{E}$ is a general member of the stack of prioritary sheaves, so it is stable.
\end{proof}

\begin{corollary} \label{Z1}
A general member of $\mathcal{E}$ is slope-stable for $d\leq g-1$.
\end{corollary}

\begin{proof} 
It suffices to exhibit one slope-stable member. By the lemma, for $d=g-1$ a general member of $\mathcal{E}$ is a general member of $M(r,\frac{g-1}{r},\xi_r(\frac{g-1}{r}))$. Because the general member is slope-stable (\cite[Corollaire 4.12]{DLP}), any elementary modification of a general member of $\mathcal{E}$ is again slope-stable.
\end{proof}

It follows that the family $\mathcal{E}$ determines a map $Y \dashrightarrow M(\vv)$, and we set $Z_1 \subseteq M(\vv)$ to be the closure of its image.

For the second family, we first assume $\gcd(r,c_1)=1$.  Let $\vv'$ be the extremal character associated to $\vv$ in the sense of Definition \ref{extremaldef}. Because $\mu(\vv'_{CH}) \leq \mu(\vv')$, we have $\mu(\vv') > 1/3$. Setting $\vv'' = \vv - \vv'$ and letting $\Delta(\vv) \gg 0$, we may choose stable $E'' \in M(\vv'')$ and semistable $E' \in M(\vv')$ and form extensions
    \[0 \rightarrow E' \rightarrow E \rightarrow E'' \rightarrow 0.\]
Then $\ch(E) = \vv$ and $E$ is generically stable (see Proposition \ref{extremal}). When $r(E)$ and $c_1(E)$ are coprime, $\mu(\vv')<\mu(\vv)$, and $\vv'$ has minimal discriminant. If $r(E)$ and $c_1(E)$ are not coprime, our construction will differ slightly. Because $\mu(\vv') > 1/3$, we have $\Delta(\vv') < \xi_r(\mu(\vv'))$, i.e., $\chi(E') > r'$ (see the proof of Theorem \ref{nonempty}). Write $\chi(E') = r' + \epsilon'$, and let $\epsilon:=\min\{\epsilon', r''\}$. By Definition \ref{extremaldef}, $r''>0$, hence $\epsilon>0$. We want to examine the Brill-Noether locus $B^{r'' - \epsilon}(\vv'')$, i.e., the locus of sheaves $E'' \in M(\vv'')$ with $h^0(E'') \geq r'' - \epsilon$. Our second family will consist of extension sheaves $E$ as above, where $E'$ has general cohomology $h^0(E') = r' + \epsilon'$ and $E''$ has potentially special cohomology $h^0(E'') = r'' - \epsilon$. If we show that the appropriate Brill-Noether locus is nonempty, we may construct a projective bundle $\PP$ over $M(\vv') \times B^{r''-\epsilon}(\vv'')$ supporting the family of sheaves $E$ as above. Then the family determines a map $\PP \dashrightarrow M(\vv)$; let $Z_2$ be the closure of the image. 

\begin{proof}[Proof of Theorem \ref{comps}]
We have as above rational maps $Y \dashrightarrow Z_1$ and $\PP \dashrightarrow Z_2$ into $M(\vv)$. We will show that $Z_2$ is well-defined (i.e., the Brill-Noether locus $B^{r''-\epsilon}(\vv'')$ is nonempty), that $Z_1$ is an irreducible component of $B^r(\vv)$, and $\dim(Z_2) > \dim(Z_1)$.

We now consider the first family of sheaves. By construction, the map $Y \dashrightarrow Z_1$ has fiber $\Aut(\OO^r)$. Recall that we set $\Delta = \Delta_0 + k/r$, where $\Delta_0$ is minimal and $k \geq 0$. We have
    \begin{align*}
        \dim Z_1 &= \dim Y - \dim \Aut(\OO^r) \\
        &= \dim \Pic^d_{\mathcal{C}/U} + r \cdot \ext^1(\OO_C(D), \OO) - r^2 \\
        &= c_1^2 + 1 + r(3c_1 + k) - r^2 \\
        &= rk + (\text{const}).
    \end{align*}

Let $V \subseteq B^r(\vv)$ be an irreducible subvariety containing $Z_1$, and let $E \in Z_1$ be general. Because $h^0(E) = r$ and the evalutation map $\ev_E$ for $E$ has full rank, we can shrink $V$ to assume that each member of $V$ does as well. In particular, the map $\psi$ constructed in the proof of Theorem \ref{depth} carries $E$ to a smooth curve. It follows that the general point of $V$ is as well, and the inverse image $\psi|_V^{-1}(U) \subseteq V$ is an open dense subset; in particular, $V \subseteq Z_1$. We conclude finally that $Z_1$ is an irreducible component of $B^r(\vv)$.

We now consider the second family of sheaves. We first assume that $\gcd(r(E), c_1(E)) = 1$. By Definition \ref{extremaldef}, we have $r'' > 0$. We conclude from Theorem \ref{nonempty} that the Brill-Noether locus $B^{r''}(\vv'')$ is nonempty, so in particular the larger Brill-Noether locus $B^{r''-\epsilon}(\vv'')$ is as well. If $r'=1$, then $E'$ is a line bundle since it has minimal discriminant, and $H^1(E')=0$. If $r' \geq 2$, by Theorem \ref{GH}, $H^1(E)=0$ since by construction $E'$ is chosen to be general. So for generic $E$ as in the construction of $Z_2$ we have 
    \[h^0(E) \geq (r'+\epsilon')+(r''-\epsilon)\geq r(E),\]
as desired.

We compute the dimension of $Z_2$. As usual, write $\Delta(\vv'') = \Delta_0 + k/r''$, where $\Delta_0$ is the minimal discriminant of a stable bundle of slope $\mu(\vv'')$ and rank $r''$. The expected dimension of $B^{r''-\epsilon}(\vv'')$ is
    \begin{align*}
        \expdim B^{r''-\epsilon}(\vv'') &= \dim M(\vv'') - \expcodim B^{r''-\epsilon}(\vv'') \\
        &= r''^2(2\Delta(\vv'') - 1) + 1 - (r'' - \epsilon)(r'' - \epsilon - \chi(\vv''))\\
        &= 2r''k  - (r''-\epsilon)k + (\text{const}_1)
    \end{align*} 
where (const$_1$) is independent of $k$. The dimension of $\PP$ is given by
    \begin{equation} \label{eq:Pdim} 
    \begin{split}
        \dim \PP &= \dim B^{r''-\epsilon}(\vv'') + \dim M(\vv') + \ext^1(E'', E') - 1 \\
        &\geq 2r''k - (r''-\epsilon)k + (\text{const}_1) + r'^2(2\Delta(\vv') -1) \\
        &= (r+\epsilon)k + (\text{const}_2).
    \end{split}
    \end{equation}
where again (const$_2$) is independent of $k$.

We claim that the fiber dimension of the map $\PP \rightarrow Z_2$ is bounded independent of $k$; for $E \in Z_2$ let $\PP_E$ denote the fiber over $E$. To show this we consider the forgetful maps $\PP_E \rightarrow \Quot(E, \vv'')$ to the Quot scheme that assigns to an extension class $0 \rightarrow E' \rightarrow E \rightarrow E'' \rightarrow 0$ the quotient $E \rightarrow E''$. The fiber of the forgetful map has dimension bounded by
    \[\dim \Aut(E') \leq \hom(E', E') \leq s^2\]
where the Jordan-H\"older filtration of $E'$ has length $s$. In particular it is independent of $k$, so it is enough to bound the dimension of the Quot scheme. There is a canonical identification $T_{[E \rightarrow E'']}\Quot(E, \vv'') \simeq \Hom(E', E'')$ and its dimension bounds the dimension of $\Quot(E,\vv'')$. Now consider a general line $L \subseteq \PP^2$ in the locally free loci for $E'$ and $E''$ and the restriction 
    \[0 \rightarrow E''(-1) \rightarrow E'' \rightarrow E''|_L \rightarrow 0\]
of $E''$ to $L$. Then $\Hom(E', E''(-1)) = 0$, and by the Grauert-M\"ulich Theorem (\cite[Theorem 3.2.1]{HL}), $\hom(E', E''|_L) = h^0(E'^{\vee} \otimes E''|_L)$ is bounded independent of $\Delta(E'')$, hence independent of $k$, and the claim follows. 

From the dimension count above, we see that
    \begin{equation} \label{eq:Z2dim}
        \dim Z_2 \geq \dim \PP - \dim \PP_E = (r + \epsilon)k + (\text{const}).
    \end{equation}

When the rank and degree fail to be coprime, let $\vv_{CH}'$ be the extremal character of $\vv$. Since $\Delta(\vv) \gg 0$, an extremal decomposition of $\vv$ exists. Denote the quotient character by $\vv_{CH}''$. Now by Lemma \ref{nonint}, $r(\vv_{CH}'')>0$. Since $\mu(\vv_{CH}') > 1/3$, $\mu(\vv_{CH}'') > 1/3$. Since $\vv_{CH}'$ has minimal discriminant, $\Delta(\vv_{CH}') \gg 0$, so Corollary \ref{Z1} implies that the general member of $Z_1(\vv_{CH}'')$ is slope-stable. The character $\vv_{CH}'$ is below $\xi_r$ (see the proof of Theorem \ref{nonempty}), so we may write $\chi(\vv_{CH}')=r'+\epsilon'$. Set $\epsilon=\min \{\epsilon', r''\}>0$. We have $Z_1(\vv_{CH}'') \subseteq B^{r''}(\vv_{CH}'')\subseteq B^{r''-\epsilon}(\vv_{CH}'')$. Let $B''$ be the component of $B^{r''-\epsilon}(\vv_{CH}'')$ that contains $Z_1(\vv_{CH}'')$. Then a general member of $B''$ is slope-stable, and since $B^{r''-\epsilon}$ is a determinantal variety, $\dim B'' \geq (r''+\epsilon)k + \text{const}''$, where const$''$ is independent of $k$ and $\Delta(\vv)= \Delta_{\min}+k/r$. By precisely the same argument as in the proof of \cite[Theorem 6.4]{CH2}, a general extension of the form 
    \[0 \rightarrow E' \rightarrow E \rightarrow E'' \rightarrow 0\]
is stable, where $E'' \in B''$ and $E' \in M(\vv')$ are general members. In particular, $h^0([E])$ is well-defined for such $[E]\in M(\vv)$ (since there is no other sheaf that is $S$-equivalent to $E$). If $r(\vv_{CH}')=1$, since $\vv_{CH}'$ has minimal discriminant, $E'$ is a line bundle and $h^1(E')=0$. If $r(\vv_{CH}')>1$, since $E'$ is general in $M(\vv')$, by Theorem \ref{GH}, $h^{0}(E')=r'+\epsilon'$ and $h^1(E')=0$. In either case, $$h^{0}(E)=(r'+\epsilon')+(r''-\epsilon)\geq (r'+\epsilon)+(r''-\epsilon)=r'+r''=r.$$
We obtain a rational map $\PP \dasharrow B^r(\vv) \subseteq M(\vv)$ where $\PP$ is the projective bundle over $M(\vv')\times B''$ whose fiber over $(E',E'')$ is $\PP \Ext^1(E'',E')$. Denote the image of this rational map by $Z_2$. Then the dimension counts (\ref{eq:Pdim}) and (\ref{eq:Z2dim}) apply, and we see that 
    \[\dim Z_2 = (r + \epsilon)k + (\text{const}) >rk + (\text{const}) = \dim Z_1\]
when $k\gg 0$. This completes the proof.
\end{proof}

\begin{example}
Let $\vv_k \in K(\PP^2)$ be the Chern character
    \[\ch(\vv) = (\ch_0(\vv), \ch_1(\vv), \ch_2(\vv)) = (3, 2, -1-k)\]
so that
    \[\mu(\vv) = \frac{2}{3}, \quad \Delta(\vv) = \frac{5}{9} + \frac{k}{3}.\]
Then $\vv_k$ is stable for all $k \geq 0$. As in Theorem \ref{comps}, there are two components in the Brill-Noether loci $B^3(\vv_k)$, whose general members are given as extensions
    \[0 \rightarrow \OO^3 \rightarrow E \rightarrow \OO_C(D) \rightarrow 0\]
with $C \subseteq \PP^2$ a smooth conic and $D$ of degree $1-k$, and
    \[0 \rightarrow T_{\PP^2}(-1) \rightarrow E \rightarrow I_Z(1) \rightarrow 0\]
where $Z \subseteq \PP^2$ is a codimension two subscheme of length $k+1$. For both extension types the general extension bundle $E$ is stable, as in Theorem \ref{comps}, and the locus $Z_1$ whose points correspond to the first extension type form an irreducible component of $B^3(\vv_k)$. Its dimension is
    \[\dim Z_1 = \dim \Pic^{1-k}_{\mathcal{C}/U} + \ext^1(\OO_C(D), \OO^3) - 1 + \dim GL_3 = 3k + 8.\]
The locus $Z_2$ corresponding to the second extension type has dimension
    \[\dim Z_2 = \dim \PP^{2[k+1]} + \ext^1(I_Z(1), T_{\PP^2}(-1)) - 1 +(\text{const})= 4k + (\text{const})\]
Thus $Z_2$ lies in an irreducible component distinct from $Z_1$, and we conclude the Brill-Noether locus $B^3(\vv_k)$ is reducible for $k\gg 0$.
\end{example}

\subsection*{Funding}
This work was supported by the National Science Foundation [grant \textit{no.} 1246844] to B.G.

\subsection*{Acknowledgements}
We are happy to thank Izzet Coskun for very many helpful conversations and his support while this work was carried out. We thank Geoffrey Smith for helpful comments on a preliminary draft of this paper. We thank the referees for many valuable suggestions on an earlier draft.

The figures in this paper were generated on Mathematica and graffitikz. graffitikz is an open-source 2D vector shape editor that supports TikZ output, created by Wenyu Jin. It can be found online at \url{https://github.com/wyjin/graffitikz}.

\bibliographystyle{alpha}
\bibliography{bibliography.bib}
\end{document}